\documentclass[11pt]{article}
\usepackage{amsmath,amssymb,amsthm}
\usepackage[all]{xy}
\usepackage{comment}

\newtheorem{theorem}{Theorem}[section]
\newtheorem{lemma}[theorem]{Lemma}
\newtheorem{proposition}[theorem]{Proposition}
\newtheorem{corollary}[theorem]{Corollary}

\theoremstyle{definition}
\newtheorem{definition}[theorem]{Definition}
\newtheorem{remark}[theorem]{Remark}
\newtheorem{example}[theorem]{Example}
\newtheorem{problem}[theorem]{Problem}

\newtheorem*{acknowledgements}{Acknowledgements}

\renewcommand{\phi}{\varphi}

\newcommand{\Aut}{\operatorname{Aut}}

\newcommand{\Hom}{\operatorname{Hom}}
\newcommand{\Homeo}{\operatorname{Homeo}}
\renewcommand{\Im}{\operatorname{Im}}
\newcommand{\id}{\operatorname{id}}
\newcommand{\Ker}{\operatorname{Ker}}
\newcommand{\rank}{\operatorname{rank}}

\newcommand{\N}{\mathbb{N}}
\newcommand{\Z}{\mathbb{Z}}
\newcommand{\Q}{\mathbb{Q}}
\newcommand{\R}{\mathbb{R}}
\newcommand{\C}{\mathbb{C}}

\newcommand{\G}{\mathcal{G}}
\renewcommand{\H}{\mathcal{H}}
\renewcommand{\S}{\mathcal{S}}

\title{Classifying Stein's groups}
\author{Hiroki Matui 
\thanks{The author was supported by JSPS KAKENHI Grant Number 23K22397.}\\
Graduate School of Science \\
Chiba University \\
Inage-ku, Chiba 263-8522, Japan}
\date{}

\begin{document}
\maketitle

\begin{abstract}
In this paper, 
we provide a comprehensive classification of Stein's groups, 
which generalize the well-known Higman-Thompson groups. 
Stein's groups are defined as groups of piecewise linear bijections 
of an interval with finitely many breakpoints and slopes 
belonging to specified additive and multiplicative subgroups 
of the real numbers. 
Our main result establishes a classification theorem 
for these groups under the assumptions that 
the slope group is finitely generated and 
the additive group has rank at least 2. 
We achieve this by interpreting Stein's groups 
as topological full groups of ample groupoids. 
A central concept in our analysis is the notion of $H^1$-rigidity 
in the cohomology of groupoids. 
In the case where the rank of the additive group is 1, 
we adopt a different approach using attracting elements 
to impose strong constraints on the classification. 
\end{abstract}

\section{Introduction}

Stein's groups are a natural and profound generalization of 
Thompson's group $V$, 
which was first introduced by R. J. Thompson in the 1960s. 
Thompson's group $V$ is famously recognized as 
the first known example of an infinite, finitely presented simple group, 
a discovery that sparked significant interest 
in the study of groups with these unique properties. 
This work was extended by G. Higman \cite{MR0376874}, 
who generalized Thompson's group to an infinite family of groups 
now known as the Higman-Thompson groups $V_{n,r}$, 
which exhibit similar structural properties. 
Expanding on this foundation, 
M. Stein \cite{St92TAMS} introduced a broader class of groups, 
now referred to as Stein's groups. 
A Stein's group $V(\Gamma,\Lambda,\ell)$ is defined 
as a group of piecewise linear bijections of an interval $[0,\ell)$ 
with finitely many breakpoints in $\Gamma$ and slopes in $\Lambda$, 
where $\Gamma$ is an additive subgroup of real numbers and 
$\Lambda$ is a multiplicative subgroup of positive real numbers 
(see Definition \ref{Steingroup} for the precise definition). 
The Higman-Thompson group $V_{n,r}$ corresponds to 
choosing $\Lambda=\langle n\rangle$, $\Gamma=\Z[1/n]$ and $\ell=r$. 
Stein's groups generalize the Higman-Thompson group 
while retaining key properties like the simplicity of derived subgroups, 
and they provide a rich framework for studying algebraic features 
such as simplicity and finiteness, 
expanding the understanding of infinite groups. 
Nevertheless 
there has been no systematic study on the classification of them. 

In this paper, 
we provide a complete classification of Stein's groups 
under the assumptions that $\Lambda$ is finitely generated and 
$\Gamma$ has rank at least 2 
(Theorem \ref{classifygroupoid} and Theorem \ref{classifygroup}). 
The key method involves interpreting Stein's groups 
as topological full groups of ample groupoids. 
A central tool in this approach is the isomorphism theorem, 
which states that the isomorphism of topological full groups 
and ample groupoids are equivalent (Theorem \ref{isomorphism}). 
This theorem provides a powerful framework for our classification. 
This approach builds on O. Tanner's earlier work \cite{Ta2312arXiv}, 
in which he determined 
when the derived subgroup of Stein's group is finitely generated. 
In this paper, we designate as ``Stein groupoid'' 
the ample groupoid that realizes Stein's group 
as its topological full group (Definition \ref{Steingroupoid}). 
The classification of Stein's groups then reduces to 
classifying these Stein groupoids (Lemma \ref{groupVSgroupoid}). 
Under the given assumptions, it will be shown that 
the cohomology group $H^1$ of a Stein groupoid is naturally isomorphic 
to the $H^1$ of $\Gamma\rtimes\Lambda$ (Proposition \ref{Sisrigid}). 
In such a case, we say that 
the action $\Gamma\rtimes\Lambda\curvearrowright\R_\Gamma$ 
is $H^1$-rigid (Definition \ref{H1rigidity}). 
This result serves as a stepping stone 
towards determining the isomorphism classes of Stein groupoids, 
and by extension, the classification of Stein's groups. 
Although several results are known 
concerning the homology groups of Stein groupoids 
\cite{CPPR11JFA,Li15JFA,Ta2312arXiv}, 
this paper distinguishes itself by focusing on cohomology. 
It is worth mentioning that 
our $H^1$-rigidity implies continuous $\Z$-cocycle rigidity 
in the sense of X. Li \cite{Li18ETDS}. 
He obtained many interesting examples of 
continuous orbit equivalence rigidity 
by using the notion of continuous cocycle rigidity. 
Indeed our main result can be also regarded 
as a new continuous orbit equivalence rigidity result 
(Remark \ref{OErigidity}). 
Furthermore, 
it is important to note that when the rank of $\Gamma$ is $1$, 
the current method does not apply. 
This is because the action $\Gamma\rtimes\Lambda\curvearrowright\R_\Gamma$ 
fails to be $H^1$-rigid, 
requiring a different approach to address this case. 
Existing techniques do not seem to resolve this issue effectively. 
On the other hand, 
by utilizing the concept of attracting elements, 
we observe that for two Stein groupoids to be isomorphic, 
there must be order-preserving embeddings 
between the groups of slopes (Lemma \ref{orderpreserve}). 
This observation applies even in the case where the rank of $\Gamma$ is $1$ 
and provides strong constraints on when Stein groupoids, 
and consequently Stein's groups, can be isomorphic. 
Thus, this offers valuable insight 
into the classification of Stein groups, even in the rank $1$ case. 

This paper is organized as follows. 
In Section 2, we provide the necessary preliminaries 
on ample groupoids and topological full groups, 
establishing the foundation for the subsequent discussions. 
We also demonstrate that 
$H^1$-rigidity is preserved under semi-direct products by $\Z$. 
Section 3 introduces Stein groupoids and 
illustrates how Stein's groups can be viewed as 
topological full groups of them. 
Also, we observe that if two Stein groupoids are isomorphic, 
then their slope groups must be isomorphic, at least as abelian groups 
(Lemma \ref{isotropy}). 
In Section 4, we prove that 
the action $\Gamma\curvearrowright\R_\Gamma$ is $H^1$-rigid, 
provided the rank of $\Gamma$ is at least $2$. 
Section 5 builds on this result to classify Stein groupoids 
by identifying the natural homomorphism from a Stein groupoid 
to $\Lambda$ in terms of the cohomology of ample groupoids, 
using the $H^1$-rigidity of the $\Gamma\rtimes\Lambda$-action. 
In Section 6, 
we study the case where 
$\Gamma=\Z[1/n]$ and $\Lambda=\langle n\rangle$, 
denoted $\S_n:=\S(\Z[1/n],\langle n\rangle)$, 
which is not $H^1$-rigid. 
We also explore an interesting phenomenon 
where $\S_2$ can be embedded into another Stein groupoid, 
suggesting a potentially intriguing problem for further investigation. 
Finally, Section 7 discusses the use of attracting elements 
in providing strong constraints on when two Stein groupoids, 
and consequently Stein groups, can be isomorphic, 
even in the case of rank $1$. 

\begin{acknowledgements}
I would like to thank O. Tanner and X. Li 
for very helpful comments and discussions. 
This work was supported by 
the Research Institute for Mathematical Sciences, 
an International Joint Usage/Research Center located 
in Kyoto University. 
\end{acknowledgements}

\section{Preliminaries}

\subsection{Ample groupoids}

The cardinality of a set $A$ is written $\#A$ and 
the characteristic function of $A$ is written $1_A$. 
We say that a subset of a topological space is clopen 
if it is both closed and open. 
A topological space is said to be totally disconnected 
if its connected components are singletons. 
By a Cantor set, 
we mean a compact, metrizable, totally disconnected space 
with no isolated points. 
It is known that any two such spaces are homeomorphic. 

In this article, by an \'etale groupoid 
we mean a second countable locally compact Hausdorff groupoid 
such that the range map is a local homeomorphism. 
(We emphasize that our \'etale groupoids are always assumed to be Hausdorff, 
while non-Hausdorff groupoids are also studied actively.) 
We refer the reader to \cite{Re_text,R08Irish} 
for background material on \'etale groupoids. 
For an \'etale groupoid $\G$, 
we let $\G^{(0)}$ denote the unit space and 
let $s$ and $r$ denote the source and range maps, 
i.e.\ $s(g)=g^{-1}g$, $r(g)=gg^{-1}$. 
A subset $U\subset\G$ is called a bisection 
if $r|U$ and $s|U$ are injective. 
The \'etale groupoid $\G$ has a basis for its topology 
consisting of open bisections. 
For $x\in\G^{(0)}$, 
the set $r(s^{-1}(x))$ is called the orbit of $x$. 
When every orbit is dense in $\G^{(0)}$, $\G$ is said to be minimal. 
For a subset $Y\subset\G^{(0)}$, 
the reduction of $\G$ to $Y$ is $r^{-1}(Y)\cap s^{-1}(Y)$ and 
denoted by $\G|Y$. 
If $Y$ is open, then 
the reduction $\G|Y$ is an \'etale subgroupoid of $\G$ in an obvious way. 
An \'etale groupoid $\G$ is called ample 
if its unit space $\G^{(0)}$ is totally disconnected. 
An \'etale groupoid is ample if and only if 
it has a basis for its topology consisting of compact open bisections. 
A (non-zero) Borel measure $\mu$ on $\G^{(0)}$ is 
said to be $\G$-invariant 
if for every compact bisection $U\subset\G$, 
one has $\mu(r(U))=\mu(s(U))<\infty$. 

From a group action on a topological space, 
we can form a transformation groupoid. 
Let $\phi:\Gamma\curvearrowright X$ be an action of 
a countable discrete group $\Gamma$ 
on a locally compact Hausdorff space $X$. 
The transformation groupoid $\G:=X\rtimes_\phi\Gamma$ is 
$X\times\Gamma$ equipped with the product topology. 
The unit space of $\G$ is given by $\G^{(0)}=X\times\{1\}$ 
(where $1$ is the identity of $\Gamma$), 
with range and source maps 
$r(x,\gamma)=(x,1)$ and $s(x,\gamma)=(\phi_\gamma^{-1}(x),1)$. 
The unit space $\G^{(0)}$ is often identified with $X$. 
Multiplication is given by 
$(x,\gamma)\cdot(x',\gamma')=(x,\gamma\gamma')$, 
and the inverse of $(x,\gamma)$ is $(\phi_\gamma^{-1}(x),\gamma^{-1})$. 
Such a transformation groupoid is always \'etale. 
It is ample if and only if $X$ is totally disconnected. 
Moreover, $\G$ is minimal if and only if $\phi$ is minimal 
i.e.\ for all $x\in X$, 
the orbit $\{\phi_\gamma(x)\mid\gamma\in\Gamma\}$ is dense in $X$. 
A Borel measure $\mu$ on $\G^{(0)}$ is $\G$-invariant 
if and only if $\mu$ is $\Gamma$-invariant 
under the identification of $\G^{(0)}$ and $X$.

\subsection{Topological full groups}

In this subsection, 
we recall the notion of topological full groups. 
For $x\in\G^{(0)}$, 
we write $\G_x=r^{-1}(x)\cap s^{-1}(x)$ and call it the isotropy group of $x$. 
The isotropy bundle of $\G$ is 
$\G'=\{g\in\G\mid r(g)=s(g)\}=\bigcup_{x\in\G^{(0)}}\G_x$. 
We say that $\G$ is essentially principal 
if the interior of $\G'$ is $\G^{(0)}$. 

\begin{definition}[{\cite[Definition 2.3]{Ma12PLMS}}]
Let $\G$ be an essentially principal ample groupoid 
whose unit space $\G^{(0)}$ is compact. 
The set of all $\alpha\in\Homeo(\G^{(0)})$, 
for which there exists a compact open bisection $U\subset\G$ 
such that $r(U)=\G^{(0)}=s(U)$ and 
$r(g)=\alpha(s(g))$ holds for all $g\in U$, 
is called the topological full group of $\G$ and is denoted by $[[\G]]$. 
\end{definition}

For $\alpha\in[[\G]]$ the compact open bisection $U$ as above uniquely exists, 
because $\G$ is essentially principal. 
Obviously $[[\G]]$ is a subgroup of $\Homeo(\G^{(0)})$. 
Since $\G$ is second countable, it has countably many compact open subsets, 
and so $[[\G]]$ is at most countable. 
A homeomorphism $\alpha:\G^{(0)}\to\G^{(0)}$ belongs to $[[\G]]$ 
if and only if for any $x\in\G^{(0)}$ 
there exists a compact open bisection $V$ 
such that $x$ is in $s(V)$ and 
$\alpha$ equals $r\circ(s|V)^{-1}$ on a neighborhood of $x$. 
Thus, $\alpha$ is in $[[\G]]$ if and only if 
the `graph' of $\alpha$ is a clopen subset of $\G$. 

The following theorem says that 
the topological full group $[[\G]]$ remembers the ample groupoid $\G$. 
For a group $G$, its derived subgroup is written $D(G)$. 

\begin{theorem}
[{\cite[Theorem 0.2]{Ru89TAMS}, \cite[Theorem 3.10]{Ma15crelle}}]
\label{isomorphism}
For $i=1,2$, let $\G_i$ be an essentially principal ample groupoid 
whose unit space is compact. 
Suppose that $\G_i$ is minimal. 
The following conditions are equivalent. 
\begin{enumerate}
\item $\G_1$ and $\G_2$ are isomorphic as \'etale groupoids. 
\item $[[\G_1]]$ and $[[\G_2]]$ are isomorphic as discrete groups. 
\item $D([[\G_1]])$ and $D([[\G_2]])$ are isomorphic as discrete groups. 
\end{enumerate}
\end{theorem}

\subsection{Homology groups and cohomology groups}

Let $A$ be a topological abelian group. 
For a locally compact Hausdorff space $X$, 
let $C(X,A)$ be the set of $A$-valued continuous functions. 
We denote by $C_c(X,A)\subset C(X,A)$ 
the subset consisting of functions with compact support. 
With pointwise addition, 
$C(X,A)$ and $C_c(X,A)$ are abelian groups. 

Let $\pi:X\to Y$ be a local homeomorphism 
between locally compact Hausdorff spaces. 
For $f\in C_c(X,A)$, we define a map $\pi_*(f):Y\to A$ by 
\[
\pi_*(f)(y):=\sum_{\pi(x)=y}f(x). 
\]
It is not so hard to see that $\pi_*(f)$ belongs to $C_c(Y,A)$ and 
that $\pi_*$ is a homomorphism from $C_c(X,A)$ to $C_c(Y,A)$. 
Besides, if $\pi':Y\to Z$ is another local homeomorphism to 
a locally compact Hausdorff space $Z$, then 
one can check $(\pi'\circ\pi)_*=\pi'_*\circ\pi_*$ in a direct way. 
Thus, $C_c(\cdot,A)$ is a covariant functor 
from the category of locally compact Hausdorff spaces 
with local homeomorphisms 
to the category of abelian groups with homomorphisms. 

Let $\G$ be an ample groupoid. 
For $n\in\N$, we write $\G^{(n)}$ 
for the space of composable strings of $n$ elements in $\G$, that is, 
\[
\G^{(n)}:=\{(g_1,g_2,\dots,g_n)\in\G^n\mid
s(g_i)=r(g_{i+1})\text{ for all }i=1,2,\dots,n{-}1\}. 
\]
For $n\geq2$ and $i=0,1,\dots,n$, 
we let $d_i^{(n)}:\G^{(n)}\to\G^{(n-1)}$ be a map defined by 
\[
d_i^{(n)}(g_1,g_2,\dots,g_n):=\begin{cases}
(g_2,g_3,\dots,g_n) & i=0 \\
(g_1,\dots,g_ig_{i+1},\dots,g_n) & 1\leq i\leq n{-}1 \\
(g_1,g_2,\dots,g_{n-1}) & i=n. 
\end{cases}
\]
When $n=1$, we let $d_0^{(1)},d_1^{(1)}:\G^{(1)}\to\G^{(0)}$ be 
the source map and the range map, respectively. 
Clearly the maps $d_i^{(n)}$ are local homeomorphisms. 
Define the homomorphisms $\partial_n:C_c(\G^{(n)},A)\to C_c(\G^{(n-1)},A)$ 
by 
\[
\partial_n:=\sum_{i=0}^n(-1)^id^{(n)}_{i*}. 
\]
The abelian groups $C_c(\G^{(n)},A)$ 
together with the boundary operators $\partial_n$ form a chain complex. 
When $n=0$, we let $\partial_0:C_c(\G^{(0)},A)\to0$ be the zero map. 

\begin{definition}
[{\cite[Section 3.1]{CM00crelle}, \cite[Definition 3.1]{Ma12PLMS}}]
\label{homology}
For $n\geq0$, we let $H_n(\G,A)$ be the homology groups of 
the chain complex above, 
i.e.\ $H_n(\G,A):=\Ker\partial_n/\Im\partial_{n+1}$, 
and call them the homology groups of $\G$ with constant coefficients $A$. 
When $A=\Z$, we simply write $H_n(\G):=H_n(\G,\Z)$. 
In addition, we call elements of $\Ker\partial_n$ cycles and 
elements of $\Im\partial_{n+1}$ boundaries. 
For a cycle $f\in\Ker\partial_n$, 
its equivalence class in $H_n(\G,A)$ is denoted by $[f]$. 
\end{definition}

Next, we introduce cohomology groups of $\G$. 
For $n\geq0$, 
we define the homomorphisms $\delta^n:C(\G^{(n)},A)\to C(\G^{(n+1)},A)$ by
\[
\delta^n(\xi):=\sum_{i=0}^{n+1}(-1)^i(\xi\circ d^{(n+1)}_i)
\]
for $\xi\in C(\G^{(n)},A)$. 
We let $\delta^{-1}:0\to C(\G^{(0)},A)$ be the zero map. 
The abelian groups $C(\G^{(n)},A)$ 
together with the coboundary operators $\delta^n$ form a cochain complex. 

\begin{definition}[{\cite{Re_text}}]\label{cohomology}
For $n\geq0$, we let $H^n(\G,A)$ be the cohomology groups of 
the cochain complex above, 
i.e.\ $H^n(\G,A):=\Ker\delta^n/\Im\delta^{n-1}$, 
and call them the cohomology groups of $\G$ with constant coefficients $A$. 
When $A=\Z$, we simply write $H^n(\G):=H^n(\G,\Z)$. 
In addition, we call elements of $\Ker\delta^n$ cocycles and 
elements of $\Im\delta^{n-1}$ coboundaries. 
For a cocycle $\xi\in\Ker\delta^n$, 
its equivalence class in $H^n(\G,A)$ is denoted by $[\xi]$. 
\end{definition}

Let $\pi:\G\to\H$ be a homomorphism between ample groupoids. 
We let $\pi^{(0)}:\G^{(0)}\to\H^{(0)}$ denote 
the restriction of $\pi$ to $\G^{(0)}$, 
and $\pi^{(n)}:\G^{(n)}\to\H^{(n)}$ denote 
the restriction of the $n$-fold product $\pi\times\pi\times\dots\times\pi$ 
to $\G^{(n)}$. 
It is easy to see that 
$\pi^{(n)}$ commute with the maps $d^{(n)}_i$. 
Hence we obtain homomorphisms 
\[
H^n(\pi):H^n(\H,A)\to H^n(\G,A). 
\]
Moreover, 
when the homomorphism $\pi$ is a local homeomorphism 
(i.e.\ an \'etale map), 
we can consider homomorphisms 
$\pi^{(n)}_*:C_c(\G^{(n)},A)\to C_c(\H^{(n)},A)$. 
Since they commute with the boundary operators $\partial_n$, 
we obtain homomorphisms 
\[
H_n(\pi):H_n(\G,A)\to H_n(\H,A). 
\]
In \cite{MM2411arXiv}, the cup products 
\[
\smile\ :H^n(\G,\Z)\times H^m(\G,A)\to H^{n+m}(\G,A)
\]
and the cap products 
\[
\frown\ :H_n(\G,\Z)\times H^m(\G,A)\to H_{n-m}(\G,A)
\]
have been introduced for ample groupoids. 
For later use, 
we recall the following proposition from \cite{MM2411arXiv}. 

\begin{proposition}
[{\cite[Proposition 3.2]{MM2411arXiv}}]\label{capproduct}
Let $\pi:\G\to\H$ be a homomorphism between ample groupoids. 
Suppose that $\pi$ is a local homeomorphism. 
Let $m,n\in\N$ be such that $m\leq n$. 
Then, for any $[f]\in H_n(\G,\Z)$ and $[\xi]\in H^m(\H,A)$, 
one has 
\[
H_n(\pi)([f])\frown[\xi]
=H_{n-m}(\pi)\left([f]\frown H^m(\pi)([\xi])\right). 
\]
\end{proposition}

\subsection{Semi-direct products and skew products}

We recall from \cite{Re_text} 
the notion of semi-direct products and skew products of ample groupoids. 

Let $\G$ be an ample groupoid and let $\Gamma$ be a countable discrete group. 
When $\phi:\Gamma\curvearrowright\G$ is an action of $\Gamma$ on $\G$, 
the semi-direct product $\G\rtimes_\phi\Gamma$ is $\G\times\Gamma$ 
equipped with the product topology. 
The unit space of $\G\rtimes_\phi\Gamma$ is given by $\G^{(0)}\times\{1\}$ 
(where $1$ is the identity of $\Gamma$), 
with range and source maps 
$r(g,\gamma)=(r(g),1)$ and $s(g,\gamma)=(\phi_\gamma^{-1}(s(g)),1)$. 
Multiplication is given by 
$(g,\gamma)\cdot(g',\gamma')=(g\phi_\gamma(g'),\gamma\gamma')$, 
and $(g,\gamma)^{-1}=(\phi_\gamma^{-1}(g^{-1}),\gamma^{-1})$. 
There exists a natural homomorphism 
$\pi:\G\rtimes_\phi\Gamma\to\Gamma$ defined by $\pi(g,\gamma):=\gamma$. 
A transformation groupoid $X\rtimes_\phi\Gamma$ is 
an example of semi-direct products. 

When $\xi:\G\to\Gamma$ is a homomorphism, 
the skew product $\G\times_\xi\Gamma$ is $\G\times\Gamma$ 
equipped with the product topology. 
The unit space of $\G\times_\xi\Gamma$ is given by $\G^{(0)}\times\Gamma$, 
with range and source maps 
$r(g,\gamma)=(r(g),\gamma)$ and $s(g,\gamma)=(s(g),\gamma\xi(g))$. 
Multiplication is given by 
$(g,\gamma)\cdot(g',\gamma\xi(g))=(gg',\gamma)$, 
and $(g,\gamma)^{-1}=(g^{-1},\gamma\xi(g))$. 
We can define an action of $\Gamma$ on $\G\times_\xi\Gamma$ 
by $\gamma\cdot(g',\gamma'):=(g',\gamma\gamma')$. 

\begin{proposition}
Let $\G$ be an ample groupoid and 
let $\xi:\G\to\Z$ be a homomorphism. 
Let $\hat\xi\in\Aut(\G\times_\xi\Z)$ be an automorphism 
on the skew product $\G\times_\xi\Z$ 
defined by $\hat\xi(g,i):=(g,i+1)$. 
Then, there exists a long exact sequence 
\[
\xymatrix@M=8pt{
0  \ar[r] & H^0(\G) \ar[r] & 
H^0(\G\times_\xi\Z) \ar[r]^{\id-H^0(\hat\xi)} & 
H^0(\G\times_\xi\Z) \ar[r] & 
H^1(\G) \ar[r] & \cdots \\
\cdots \ar[r] & 
H^n(\G) \ar[r] & 
H^n(\G\times_\xi\Z) \ar[r]^{\id-H^n(\hat\xi)} & 
H^n(\G\times_\xi\Z) \ar[r] & 
H^{n+1}(\G) \ar[r] & \cdots . 
}
\]
\end{proposition}

\begin{proof}
Let $C^\bullet(\cdot,\Z)$ denote the cochain complex 
defining cohomology groups of groupoids 
(Definition \ref{cohomology}). 
In the same way as the proof of \cite[Lemma 1.3]{Or20JNG}, 
\[
\xymatrix@M=8pt{
0 \ar[r] & C^\bullet(\G,\Z) \ar[r] & 
C^\bullet(\G\times_\xi\Z,\Z) \ar[r]^{\id-\hat\xi} &
C^\bullet(\G\times_\xi\Z,\Z) \ar[r] & 0
}
\]
is a short exact sequence of cochain complexes, 
which yields the desired long exact sequence. 
\end{proof}

\begin{corollary}\label{LES}
Let $\H$ be an ample groupoid 
and let $\theta:\H\to\H$ be an automorphism. 
Then, there exists a long exact sequence 
\[
\xymatrix@M=8pt{
0  \ar[r] & H^0(\H\rtimes_\theta\Z) \ar[r] & 
H^0(\H) \ar[r]^{\id-H^0(\theta)} & 
H^0(\H) \ar[r] & 
H^1(\H\rtimes_\theta\Z) \ar[r] & \cdots \\
\cdots \ar[r] & 
H^n(\H\rtimes_\theta\Z) \ar[r] & 
H^n(\H) \ar[r]^{\id-H^n(\theta)} & 
H^n(\H) \ar[r] & 
H^{n+1}(\H\rtimes_\theta\Z) \ar[r] & \cdots 
}
\]
where $\H\rtimes_\theta\Z$ is the semi-direct product. 
\end{corollary}

\begin{proof}
Define a homomorphism $\xi:\H\rtimes_\theta\Z\to\Z$ 
by $\xi(g,k):=k$. 
We apply the proposition above to $\H\rtimes_\theta\Z$ and $\xi$. 
For any $x\in\H^{(0)}$ and $k\in\Z$, 
$r(x,k,0)=(x,0,0)$ and $s(x,k,0)=(\theta^{-k}(x),0,k)$. 
Hence $\H^{(0)}\times\{0\}\times\{0\}$ is full 
in $(\H\rtimes_\theta\Z)\times_\xi\Z$. 
Since $\H$ is isomorphic to 
the reduction of $(\H\rtimes_\theta\Z)\times_\xi\Z$ 
to $\H^{(0)}\times\{0\}\times\{0\}$, 
one has $H^*((\H\rtimes_\theta\Z)\times_\xi\Z)\cong H^*(\H)$. 
Moreover, 
this isomorphism intertwines $H^*(\hat\xi)$ and $H^*(\theta)$, 
and so the conclusion follows. 
\end{proof}

The following definition plays a central role 
in the later sections. 

\begin{definition}\label{H1rigidity}
Let $G\curvearrowright X$ be an action of 
a countable discrete group $G$ 
on a totally disconnected locally compact Hausdorff space $X$. 
We say that $G\curvearrowright X$ is $H^1$-rigid 
if the natural homomorphism $\pi:X\rtimes G\to G$ induces 
an isomorphism between $H^1(X\rtimes G)$ and $H^1(G)=\Hom(G,\Z)$. 
\end{definition}

\begin{remark}
In the setting of measurable dynamical systems, 
such a rigidity property is called cocycle superrigidity 
(see \cite{MR776417,Po07Invent} for instance), 
and it has led a lot of orbit equivalence rigidity results. 
In the setting of topological dynamical systems, 
the notion of continuous cocycle rigidity was introduced 
by Li in \cite[Definition 4.3]{Li18ETDS}, and 
several examples for continuous orbit equivalence rigidity were obtained. 
A topological dynamical system $G\curvearrowright X$ is said to be 
continuous $H$-cocycle rigid 
if for every continuous $H$-cocycle $\xi:X\rtimes G\to H$, 
there exists a group homomorphism $\zeta:G\to H$ 
such that $\xi$ is cohomologous to $\zeta\circ\pi$, 
where $H$ is a discrete (possibly non-abelian) group. 
Thus, $G\curvearrowright X$ is continuous $\Z$-cocycle rigid 
if the homomorphism $H^1(G)\to H^1(X\rtimes G)$ is surjective. 
Our notion of $H^1$-rigidity is almost the same 
as continuous $\Z$-cocycle rigidity of \cite{Li18ETDS}, 
but we require $H^1(G)\to H^1(X\rtimes G)$ is also injective 
in order to let the later arguments proceed smoothly. 
\end{remark}

The following proposition says that 
$H^1$-rigidity is preserved under taking a semi-direct product by $\Z$. 

\begin{proposition}\label{H1ofsemidirect}
Let $\Gamma$ be a countable discrete group and 
let $\Gamma\rtimes_\theta\Z$ be a semi-direct product 
with respect to $\theta\in\Aut(\Gamma)$. 
Let $\phi:\Gamma\rtimes_\theta\Z\curvearrowright X$ be an action 
on a totally disconnected locally compact Hausdorff space $X$. 
Suppose that $\phi:\Gamma\curvearrowright X$ is minimal and $H^1$-rigid. 
Then $\phi:\Gamma\rtimes_\theta\Z\curvearrowright X$ is also $H^1$-rigid. 
\end{proposition}

\begin{proof}
Set $\H:=X\rtimes_\phi\Gamma$ and 
$\G:=X\rtimes_\phi(\Gamma\times_\theta\Z)$. 
Let $\pi:\H\to\Gamma$ and $\pi':\G\to\Gamma\rtimes_\theta\Z$ be 
the natural homomorphisms. 
There exists an automorphism $\tilde\theta:\H\to\H$ such that 
$\theta\circ\pi=\pi\circ\tilde\theta$ and 
$\H\rtimes_{\tilde\theta}\Z$ is isomorphic to $\G$. 
Using Corollary \ref{LES}, we get the following commutative diagram: 
\[
\xymatrix@M=8pt{
0 \ar[r] & \Z \ar[r] \ar@{=}[d] & 
H^1(\Gamma\rtimes_\theta\Z) \ar[r] \ar[d]_{H^1(\pi')} & 
H^1(\Gamma) \ar[r]^{\id-H^1(\theta)} \ar[d]_{H^1(\pi)}^{\cong} & 
H^1(\Gamma) \ar[d]_{H^1(\pi)}^{\cong} \\
0 \ar[r] & \Z \ar[r] & H^1(\H\rtimes_{\tilde\theta}\Z) \ar[r] & 
H^1(\H) \ar[r]^{\id-H^1(\tilde\theta)} & H^1(\H), 
}
\]
where the horizontal sequences are exact. 
Notice that $H^0(\H)$ is $\Z$, because $\H$ is minimal. 
Since $H^1(\pi):H^1(\Gamma)\to H^1(\H)$ is an isomorphism, 
so is $H^1(\pi')$. 
\end{proof}

The following lemma will be used in Section 5. 
See \cite[Definition 3.1]{CRS17PAMS} (or \cite[Definition 4.1]{Ma12PLMS}) 
for the definition of Kakutani equivalence. 

\begin{lemma}\label{Kakutani}
Let $\G$ be an ample groupoid and 
let $A$ be a countable discrete abelian group. 
Let $\xi,\eta:\G\to A$ be homomorphisms. 
Suppose that 
$\xi$ and $\eta$ have the same cohomology class in $H^1(\G,A)$. 
\begin{enumerate}
\item The skew products $\G\times_\xi A$ and $\G\times_\eta A$ are 
isomorphic. 
\item When $\G\times_\xi A$ is minimal, 
the groupoids $\Ker\xi$ and $\Ker\eta$ are Kakutani equivalent. 
\end{enumerate}
\end{lemma}

\begin{proof}
(1)\:
There exists a continuous function $f:\G^{(0)}\to A$ such that 
$\xi(g)=\eta(g)+f(r(g))-f(s(g))$ holds for all $g\in\G$. 
Define $\pi:\G\times_\xi A\to\G\times_\eta A$ 
by $\pi(g,i):=(g,i+f(r(g))$. 
It is plain that $\pi$ is an isomorphism. 

(2)\:
Since $\G\times_\xi A$ is minimal, 
the clopen subset $\G^{(0)}\times\{0\}$ is full. 
The reduction of $\G\times_\xi A$ to $\G^{(0)}\times\{0\}$ is 
naturally isomorphic to $\Ker\xi$. 
Therefore $\G\times_\xi A$ is Kakutani equivalent to $\Ker\xi$. 
This, together with (1), implies that 
$\Ker\xi$ and $\Ker\eta$ are Kakutani equivalent. 
\end{proof}

Let us recall from \cite{Re_text} 
the notion of the essential range (asymptotic range) of a cocycle. 

\begin{definition}[{\cite[Definition 4.3]{Re_text}}]\label{essentialrange}
Let $\G$ be an ample groupoid and 
let $\xi:\G\to A$ be a homomorphism 
from $\G$ to a countable discrete abelian group $A$. 
We say that $a\in A$ is an essential value of $\xi$ 
if for any nonempty clopen subset $Y\subset\G^{(0)}$ 
there exists $g\in \G|Y$ such that $\xi(g)=a$. 
The essential range of $\xi$ 
(called asymptotic range in \cite{Re_text}) 
consists of all essential values of $\xi$. 
\end{definition}

It is well-known that 
the essential range is a subgroup of $A$ and 
depends only on the cohomology class $[\xi]\in H^1(\G,A)$ of $\xi$ 
(\cite[Proposition 4.5]{Re_text}). 
The following lemma is also folklore. 
For readers' convenience, we include a proof here. 

\begin{lemma}\label{skewisminimal}
Let $\G$ be a minimal ample groupoid and 
let $\xi:\G\to A$ be a homomorphism 
from $\G$ to a countable discrete abelian group $A$. 
The following are equivalent. 
\begin{enumerate}
\item The skew product $\G\times_\xi A$ is minimal. 
\item The subgroupoid $\Ker\xi$ is minimal and 
the essential range of $\xi$ is $A$. 
\item The subgroupoid $\Ker\xi$ is minimal and 
the range of $\xi$ is $A$. 
\end{enumerate}
\end{lemma}

\begin{proof}
(1)$\implies$(2)\:
The subgroupoid $\Ker\xi$ is canonically isomorphic to 
the reduction of $\G\times_\xi A$ to $\G^{(0)}\times\{0\}$. 
Since $\G\times_\xi A$ is minimal, 
so is $\Ker\xi$. 
Let $Y\subset\G^{(0)}$ be a nonempty clopen subset. 
Take $y\in Y$ and $a\in A$. 
As the orbit of $(y,0)$ meets $Y\times\{a\}$, 
we can find $g\in\G$ such that 
$r(g)=y$, $s(g,0)=(s(g),\xi(g))\in Y\times\{a\}$. 
This proves that $g$ is in $\G|Y$ and $\xi(g)=a$. 
Hence the essential range of $\xi$ is $A$. 

(2)$\implies$(3) is obvious. 

(3)$\implies$(1)\:
Take $(x,a)\in\G^{(0)}\times A$. 
Let $X$ be the closure of the orbit of $(x,a)$. 
Then $X$ contains $\G^{(0)}\times\{a\}$, 
because $\Ker\xi$ is minimal. 
Since $\xi:\G\to A$ is surjective, for any $b\in A$, 
there exists $g\in\G$ such that $\xi(g)=b-a$. 
By $r(g,a)=(r(g),a)$, $s(g,a)=(s(g),b)$, 
$X$ intersects with $\G^{(0)}\times\{b\}$. 
Therefore $X$ contains $\G^{(0)}\times\{b\}$ 
for all $b\in A$. 
Consequently, $\G\times_\xi A$ is minimal. 
\end{proof}

\section{Stein groupoids}

In this section, 
we introduce the notion of Stein groupoids. 
First, we recall the definition of Stein's groups 
from \cite{St92TAMS}. 

\begin{definition}[Stein's groups, {\cite{St92TAMS}}]\label{Steingroup}
Let $\Lambda\subset(0,\infty)$ be a countable multiplicative subgroup 
of the positive real numbers 
and let $\Gamma\subset\R$ be a countable $\Z\Lambda$-module, 
which is dense in $\R$. 
Let $\ell\in\Gamma\cap(0,\infty)$. 
Stein's group $V(\Gamma,\Lambda,\ell)$ is the group consisting of 
piecewise linear bijections $f$ of $[0,\ell)$ satisfying the following: 
\begin{itemize}
\item $f$ is right continuous, 
\item $f$ has finitely many discontinuous or nondifferential points, 
all in $\Gamma$, 
\item $f$ has slopes only in $\Lambda$. 
\end{itemize}
\end{definition}

For $f\in V(\Gamma,\Lambda,\ell)$, 
we can find $m\in\N$, $t_1,t_2,\dots,t_{m+1}\in\Gamma$, 
$\lambda_1,\lambda_2,\dots,\lambda_m\in\Lambda$ and 
$s_1,s_2,\dots,s_m\in\Gamma$ such that 
\[
0=t_1<t_2<\dots<t_m<t_{m+1}=\ell
\]
and 
\[
f(t)=\lambda_it+s_i\quad\forall t\in[t_i,t_{i+1}). 
\]

\begin{remark}
The family of Stein's groups encompasses 
the following interesting examples 
which have been investigated from various directions. 
\begin{enumerate}
\item The Higman-Thompson groups $V_{n,r}$ \cite{MR0376874} correspond to 
choosing $\Lambda=\langle n\rangle$, $\Gamma=\Z[1/n]$ and $\ell=r$. 
Two Higman-Thompson groups $V_{n,r}$ and $V_{m,s}$ are isomorphic 
if and only if $n=m$ and $\gcd(n{-}1,r)=\gcd(m{-}1,s)$ 
(\cite{MR0376874,MR2831934}). 
K. S. Brown \cite{Br87JPuApAl} proved that $V_{n,r}$ is of type $F_\infty$ 
by constructing an action on a certain contractible complex. 
\item In \cite{St92TAMS}, Stein studied the case that 
$\Lambda$ is generated by finitely many natural numbers 
$n_1,n_2,\dots,n_k$ and $\Gamma=\Z[1/(n_1n_2\dots n_k)]$. 
It was shown that $V(\Gamma,\Lambda,\ell)$ is also of type $F_\infty$. 
Furthermore, 
the abelianization of $V(\Gamma,\Lambda,\ell)$ was computed. 
\item In \cite{Cl95Rocky,Cl00Illinois}, S. Cleary proved that 
$V(\Z[\lambda],\langle\lambda\rangle,\ell)$ is of type $F_\infty$, 
when $\lambda$ is an algebraic integer of degree two, 
such as $\sqrt{2}{-}1$ or $(\sqrt{5}-1)/2$. 
\end{enumerate}
\end{remark}

Tanner \cite{Ta2312arXiv} observed that 
Stein's groups are realized as topological full groups of ample groupoids. 
We would like to recall this construction. 
Let $B(\R)$ denote the abelian $C^*$-algebra 
consisting of bounded Borel functions $\R\to\C$. 
For a countable dense subset $\Delta\subset\R$, 
we consider the subalgebra $A\subset B(\R)$ generated by 
\[
\{1_{[t,s)}\mid t,s\in\Delta,\ t<s\}. 
\]
Let $\R_\Delta$ be the Gelfand spectrum of $A$, 
that is, $A\cong C_0(\R_\Delta)$. 
Then, $\R_\Delta$ is 
a totally disconnected, locally compact Hausdorff space. 
Since $\Delta$ is dense in $\R$, 
we can find a continuous surjection $q:\R_\Delta\to\R$. 
One has 
\[
\#q^{-1}(x)=\begin{cases}2&x\in\Delta\\1&x\notin\Delta. \end{cases}
\]
We may and do identify $\R_\Delta$ with 
\[
\left(\R\setminus\Delta\right)\sqcup\{t_+,t_-\mid t\in\Delta\}
\]
which admits a total order such that 
$q(x)<q(y)$ implies $x<y$ and $t_-<t_+$ for all $t\in\Delta$ 
(see \cite[Definition 4.1]{Ta2312arXiv}). 
The topology on $\R_\Delta$ agrees with the order topology. 
For any $t,s\in\Delta$ with $t<s$, the interval 
\[
(t_-,s_+)=[t_+,s_-]
\]
is compact and open in $\R_\Delta$, 
and these intervals form a basis for the topology on $\R_\Delta$ 
(see \cite[Lemma 4.2]{Ta2312arXiv}). 

Let $\Lambda\subset(0,\infty)$ and $\Gamma\subset\R$ be 
as in Definition \ref{Steingroup}. 
If $\Delta$ is invariant under the translation by $\Gamma$ 
(i.e.\ $x+t\in\Delta$ for any $x\in\Delta$ and $t\in\Gamma$), 
then the group $\Gamma$ acts on $A$, and hence acts on $\R_\Delta$. 
Furthermore, 
if $\Delta$ is invariant under the multiplication by $\Lambda$ 
(i.e.\ $\lambda x\in\Delta$ 
for any $x\in\Delta$ and $\lambda\in\Lambda$), then 
the semi-direct product $\Gamma\rtimes\Lambda$ acts on $\R_\Delta$. 
In particular, 
$\Gamma\rtimes\Lambda$ naturally acts on $\R_\Gamma$. 

\begin{definition}[Stein groupoids]\label{Steingroupoid}
Let $\Lambda\subset(0,\infty)$ be a countable multiplicative subgroup 
of the positive real numbers 
and let $\Gamma\subset\R$ be a countable $\Z\Lambda$-module, 
which is dense in $\R$. 
The action $\Gamma\rtimes\Lambda\curvearrowright\R_\Gamma$ yields 
the transformation groupoid $\R_\Gamma\rtimes(\Gamma\rtimes\Lambda)$. 
We call this the Stein groupoid, 
and denote it by $\S(\Gamma,\Lambda)$. 
\end{definition}

By definition, $\S(\Gamma,\Lambda)$ is an ample groupoid. 
We may identify $\R_\Gamma$ with $\S(\Gamma,\Lambda)^{(0)}$ 
via $x\mapsto (x,0,1)$. 
Note that $\S(\Gamma,\Lambda)$ can be also regarded as 
the semi-direct product of the ample groupoid $\R_\Gamma\rtimes\Gamma$ 
by the action of $\Lambda$. 
Indeed, the $\Lambda$-action on $\R_\Gamma\rtimes\Gamma$ is given by 
\[
\lambda\cdot(x,t):=(\lambda\cdot x,\lambda t). 
\]
It is straightforward to verify that 
$\R_\Gamma\rtimes(\Gamma\rtimes\Lambda)$ is isomorphic to 
$(\R_\Gamma\rtimes\Gamma)\rtimes\Lambda$. 
Notice that the Lebesgue measure on $\R$ gives rise to 
an $\R_\Gamma\rtimes\Gamma$-invariant measure, 
and such a measure is unique up to scalar multiplication. 
The action of $\Lambda$ scales 
the $\R_\Gamma\rtimes\Gamma$-invariant measure. 

The following lemma says that 
Stein's group $V(\Gamma,\Lambda,\ell)$ is realized 
as a topological full group of the Stein groupoid $\S(\Gamma,\Lambda)$. 

\begin{lemma}[{\cite[Lemma 4.4]{Ta2312arXiv}}]
Let $\Lambda,\Gamma,\ell$ be as in Definition \ref{Steingroup}. 
Then, Stein's group $V(\Gamma,\Lambda,\ell)$ is canonically isomorphic 
to the topological full group of 
the ample groupoid $\S(\Gamma,\Lambda)|[0_+,\ell_-]$. 
\end{lemma}

This, together with Theorem \ref{isomorphism}, implies the following. 

\begin{lemma}\label{groupVSgroupoid}
For $i=1,2$, 
let $\Lambda_i,\Gamma_i,\ell_i$ be as in Definition \ref{Steingroup}. 
The following are equivalent. 
\begin{enumerate}
\item $V(\Gamma_1,\Lambda_1,\ell_1)$ is isomorphic to 
$V(\Gamma_2,\Lambda_2,\ell_2)$ as discrete groups. 
\item $D(V(\Gamma_1,\Lambda_1,\ell_1))$ is isomorphic to 
$D(V(\Gamma_2,\Lambda_2,\ell_2))$ as discrete groups. 
\item $\S(\Gamma_1,\Lambda_1)|[0_+,\ell_{1-}]$ is isomorphic to 
$\S(\Gamma_2,\Lambda_2)|[0_+,\ell_{2-}]$ as ample groupoids. 
\end{enumerate}
\end{lemma}

Thus, we have reduced the classification of Stein's groups 
to the classification of certain ample groupoids. 
More precisely, in order to classify Stein's groups, 
it suffices to determine 
when two given Stein groupoids (or their reductions) are 
isomorphic to each other. 
As a first step of this, we observe that 
the isotropy groups of $\S(\Gamma,\Lambda)$ detect 
the isomorphism class of the group $\Lambda$. 

\begin{lemma}\label{isotropy0}
Let $\S=\S(\Gamma,\Lambda)$ be a Stein groupoid and 
let $\pi:\S\to\Lambda$ be the natural homomorphism. 
\begin{enumerate}
\item For every $x\in\R_\Gamma$, the restriction of $\pi$ 
to the isotropy group $\S_x=r^{-1}(x)\cap s^{-1}(x)$ is injective. 
\item For every $x\in\{t_+,t_-\mid t\in\Gamma\}$, 
$\pi$ induces an isomorphism between $\S_x$ and $\Lambda$. 
\end{enumerate}
\end{lemma}

\begin{proof}
(1)\:
Let $(x,s,\lambda)\in\S_x\subset\S=\R_\Gamma\rtimes(\Gamma\rtimes\Lambda)$. 
We have $x=\lambda^{-1}(x-s)$. 
If $(x,s,\lambda)$ is in the kernel of $\pi$, then $\lambda=1$, 
and hence $s=0$. 
Thus, the kernel of $\pi|\S_x$ is trivial. 

(2)\:
Let $t\in\Gamma$ and $x\in\{t_+,t_-\}$. 
For any $\lambda\in\Lambda$, 
$(x,(1-\lambda)t,\lambda)\in\S$ belongs to $\S_x$ and 
$\pi(x,(1-\lambda)t,\lambda)=\lambda$. 
It follows that $\pi|\S_x:\S_x\to\Lambda$ is surjective. 
\end{proof}

\begin{lemma}\label{isotropy}
For $i=1,2$, 
let $\S_i:=\S(\Gamma_i,\Lambda_i)$ be a Stein groupoid. 
If $\S_1$ is isomorphic to $\S_2$ and 
$\Lambda_1$ is finitely generated, 
then $\Lambda_1\cong\Lambda_2$. 
\end{lemma}

\begin{proof}
Suppose that $\Phi:\S_1\to\S_2$ is an isomorphism. 
For each $i=1,2$, by the lemma above, 
there exists $x_i\in\R_{\Gamma_i}$ 
such that $(\S_i)_{x_i}\cong\Lambda_i$. 
Also, $\Phi((\S_1)_{x_1})=(\S_2)_{\Phi(x_1)}$ is isomorphic 
to a subgroup of $\Lambda_2$. 
Therefore $\Lambda_1$ is isomorphic to a subgroup of $\Lambda_2$. 
Similarly $\Lambda_2$ is isomorphic to a subgroup of $\Lambda_1$. 
As $\Lambda_1$ is finitely generated, 
one can conclude $\Lambda_1\cong\Lambda_2$. 
\end{proof}

\begin{remark}
Several results are known 
concerning the homology groups of Stein groupoids. 
When $\lambda\neq1$ is an algebraic integer, 
the homology groups of $\S(\Z[1/\lambda],\langle\lambda\rangle)$ 
can be computed in terms of the minimal polynomial of $\lambda$ 
(see Corollary 5.6 and Lemma 5.9 of \cite{Li15JFA}). 
When $\Lambda$ is generated by finitely many natural numbers 
and $\Gamma=\Z[\Lambda]$, 
the homology groups of $\S(\Gamma,\Lambda)$ was computed 
in \cite[Section 6.2]{Ta2304arXiv}. 
\end{remark}

\section{Cohomology of interval exchange transformations}

Let $\Gamma\subset\R$ be 
a countable dense (additive) subgroup. 
We let $\rank\Gamma$ denote 
the dimension of the $\Q$-vector space $\Gamma\otimes\Q$ 
and call it the rank of $\Gamma$. 
The rank is one if and only if $\Gamma$ is contained in 
$\{qt\mid q\in\Q\}$ for some $t\in\Gamma\setminus\{0\}$. 
Let $\Delta\subset\R$ be 
a countable dense subset containing $0$ 
which is invariant under the translation by $\Gamma$. 
Thus, for any $x\in\Delta$ and $t\in\Gamma$, 
we have $x+t\in\Delta$. 
The group $\Gamma$ acts on $\R_\Delta$ by translation. 
This action can be regarded 
as interval exchange transformations on $\R_\Delta$. 
Let $\mathcal{H}=\mathcal{H}(\Gamma,\Delta):=\R_\Delta\rtimes\Gamma$ be 
the associated transformation groupoid. 
In this subsection 
we would like to compute $H^1(\H)$ when $\rank\Gamma\geq2$. 

To start with, we write $\Gamma$ 
as an increasing union of finitely generated subgroups $\Gamma_n$, 
i.e.\ $\Gamma_n\subset\Gamma_{n+1}$, $\bigcup_n\Gamma_n=\Gamma$ 
and each $\Gamma_n$ is a free abelian group of finite rank. 
Let $\mathcal{H}_n:=\R_\Delta\rtimes\Gamma_n$ be 
the transformation groupoid, 
so that $\H=\bigcup_n\H_n$. 

\begin{example}
Later we will consider the case 
where a multiplicative subgroup $\Lambda\subset(0,\infty)$ 
acts on $\Gamma$. 
The simplest example is $\Lambda=\langle\lambda\rangle$, 
the group generated by a single positive real number $\lambda\neq1$, 
and $\Gamma=\Z[\Lambda]\subset\R$ (the ring generated by $\Lambda$). 
Clearly, $\rank\Gamma=\rank\Z[\lambda]=1$ 
if and only if $\lambda$ is rational, 
and $\rank\Gamma=\infty$ 
if and only if $\lambda$ is transcendental. 
Suppose that $\lambda$ is an algebraic number of degree $d$. 
Let $f(x)=a_dx^d+a_{d-1}x^{d-1}+\dots+a_1x+a_0$ be 
the minimal polynomial over $\Z$ 
such that $\gcd(a_0,a_1,\dots,a_d)=1$. 
As observed in \cite[Proposition 2.1]{CPPR11JFA}, 
one has 
\[
\Z+\Z\lambda+\dots+\Z\lambda^{d-1}\subset\Gamma
\subset\Z\left[\frac{1}{m}\right]+\Z\left[\frac{1}{m}\right]\lambda
+\dots+\Z\left[\frac{1}{m}\right]\lambda^{d-1}, 
\]
where $m:=a_da_0$. 
Hence $\rank\Gamma$ is $d$. 
By letting $\Gamma_n$ be 
the subgroup generated by $\{\lambda^k\mid\lvert k\rvert\leq n+d\}$, 
we may assume that $\Gamma_n$ is isomorphic to $\Z^d$ 
and the index of $\Gamma_n$ in $\Gamma_{n+1}$ is $m$. 
Moreover, 
if $\lambda$ and $\lambda^{-1}$ are both algebraic integers, 
then $\Gamma$ equals $\Z+\Z\lambda+\dots+\Z\lambda^{d-1}$, 
which itself is isomorphic to $\Z^d$. 
\end{example}

\begin{lemma}\label{H1ofHn}
\begin{enumerate}
\item When $\rank\Gamma_n\geq3$, 
$\Gamma_n\curvearrowright\R_\Delta$ is $H^1$-rigid. 
\item When $\rank\Gamma_n=2$, 
we have $H^1(\H_n)\cong H_0(\H_n)$. 
\end{enumerate}
\end{lemma}

\begin{proof}
The following argument is based on \cite[Proposition 5.5]{Li15JFA} 
(see also \cite[Remark 5.8]{Li15JFA} and 
\cite[Lemma 5.1]{Ta2304arXiv}), 
but we include a complete proof for the convenience of the reader. 
Let $F\subset C(\R_\Delta,\Z)$ be the subgroup 
generated by $C_c(\R_\Delta,\Z)$ and $1_{[0_+,\infty)}$. 
Then 
\[
\xymatrix@M=8pt{
0 \ar[r] & C_c(\R_\Delta,\Z) \ar[r] & F \ar[r] & 
\Z \ar[r] & 0
}
\tag{$*$}
\label{ses}
\]
is a short exact sequence of $\Z\Gamma_n$-modules. 
It is easy to see that 
$\{1_{[a_+,\infty)}\mid a\in\Delta\}$ forms 
a $\Z$-basis of $F$ on which $\Gamma_n$ acts freely. 
Hence 
\[
F\cong
\bigoplus_{\Delta/\Gamma_n}\Z\Gamma_n
\]
as $\Z\Gamma_n$-modules, 
where $\Delta/\Gamma_n$ is the quotient of $\Delta$ 
by the action of $\Gamma_n$. 
Thus 
\[
H_q(\Gamma_n,F)\cong
\begin{cases}
\bigoplus_{\Delta/\Gamma_n}\Z & q=0\\
0 & q\geq1. 
\end{cases}
\]
The homology long exact sequence attached to \eqref{ses} 
becomes 
\[
\xymatrix@M=8pt{
\cdots \ar[r] & H_2(\Gamma_n,C_c(\R_\Delta,\Z)) \ar[r] & 
0 \ar[r] & H_2(\Gamma_n) & \\
\ar[r] & H_1(\Gamma_n,C_c(\R_\Delta,\Z)) \ar[r] & 
0 \ar[r] & H_1(\Gamma_n) & \\
\ar[r] & H_0(\Gamma_n,C_c(\R_\Delta,\Z)) \ar[r] & 
\bigoplus_{\Delta/\Gamma_n}\Z \ar[r] & \Z \ar[r] & 0. \\
}
\]
which implies 
\[
H_q(\Gamma_n,C_c(\R_\Delta,\Z))\cong H_{q+1}(\Gamma_n)
\]
for all $q\geq1$. 

Suppose that $\Gamma_n$ is isomorphic to $\Z^k$. 
Choose a positive real number $l\in\Gamma_n$ 
so that $\Gamma_n/l\Z$ is torsion free. 
The groupoid $\H_n|[0_+,l_-]$ can be viewed 
as a transformation groupoid of an action of $\Z^{k-1}$ 
on the Cantor set. 
Hence we get 
\begin{align*}
H^1(\H_n)
&\cong H^1(\H_n|[0_+,l_-])\\
&\cong H_{k-1-1}(\H_n|[0_+,l_-])\\
&\cong H_{k-2}(\H_n)\\
&\cong H_{k-2}(\Gamma_n,C_c(\R_\Delta,\Z))
\end{align*}
where the second isomorphism is due to Poincar\'e duality. 
When $k=2$, the proof is completed. 
When $k\geq3$, using Poincar\'e duality again, one has 
\[
H_{k-2}(\Gamma_n,C_c(\R_\Delta,\Z))
\cong H_{k-1}(\Gamma_n)
\cong H^1(\Gamma_n), 
\]
which completes the proof. 
\end{proof}

\begin{lemma}\label{H1ofH:rank>2}
When $\rank\Gamma\geq3$, 
$\Gamma\curvearrowright\R_\Delta$ is $H^1$-rigid. 
In other words, the natural homomorphism $\H\to\Gamma$ induces 
an isomorphism between $H^1(\H)$ and $H^1(\Gamma)=\Hom(\Gamma,\Z)$. 
\end{lemma}

\begin{proof}
We may assume $\rank\Gamma_n\geq3$ for all $n\in\N$. 
Let $\xi:\H\to\Z$ be a homomorphism. 
Let $\xi_n$ be the restriction of $\xi$ to $\H_n$. 
It follows from Lemma \ref{H1ofHn} (1) that 
there exists $\eta_n\in\Hom(\Gamma_n,\Z)$ 
such that $\xi_n-\eta_n\circ\pi_n$ is a coboundary, 
where $\pi_n:\H_n\to\Gamma_n$ is the natural homomorphism. 
Let $d_n:C(\R_\Delta,\Z)\to C(\H_n,\Z)$ be 
the coboundary operator, i.e.\ 
\[
d_n(f)(g):=f(r(g))-f(s(g))\quad\forall g\in\H_n. 
\]
Define $d:C(\R_\Delta,\Z)\to C(\H,\Z)$ in the same way. 
The action of $\Gamma_n$ on $\R_\Delta$ is minimal, 
and so $\Ker d_n$ consists of constant functions. 
Therefore, for each $n\in\N$, 
there exists a unique $f_n\in C(\R_\Delta,\Z)$ such that 
$\xi_n=\eta_n\circ\pi_n+d_n(f_n)$ and $f_n(0_+)=0$. 
Since the diagram 
\[
\xymatrix@M=8pt{
H^1(\H_n) & 
H^1(\Gamma_n) \ar[l]_{\cong} & 
\Hom(\Gamma_n,\Z)  \ar[l]_{\cong} \\
H^1(\H_{n+1})  \ar[u] & 
H^1(\Gamma_{n+1}) \ar[l]_{\cong} \ar[u] & 
\Hom(\Gamma_{n+1},\Z) \ar[l]_{\cong} \ar[u] 
}
\]
is commutative, 
$\eta_n\in\Hom(\Gamma_n,\Z)$ is equal to 
the restriction of 
$\eta_{n+1}\in\Hom(\Gamma_{n+1},\Z)$ to $\Gamma_n$. 
It follows that 
there exists $\eta\in\Hom(\Gamma,\Z)$ such that 
$\eta_n$ is the restriction of $\eta$ to $\Gamma_n$. 
Moreover, for all $g\in\H_n$, one has 
\begin{align*}
d_n(f_n)(g)
&=\xi_n(g)-\eta_n(\pi_n(g))\\
&=\xi_{n+1}(g)-\eta_{n+1}(\pi_{n+1}(g))\\
&=d_{n+1}(f_{n+1})(g)=d_n(f_{n+1})(g). 
\end{align*}
Hence we obtain $f_n=f_{n+1}$. 
Thus, for all $g\in\H_n$, 
\[
\xi(g)=\xi_n(g)=\eta_n(\pi_n(g))+d_n(f_n)
=\eta(\pi(g))+d(f_1), 
\]
where $\pi:\H\to\Gamma$ is the natural homomorphism. 
Consequently, 
$H^1(\pi):H^1(\Gamma)\to H^1(\H)$ is an isomorphism. 
\end{proof}

When the rank of $\Gamma$ is two, 
$H^1(\H_n)$ may be larger than $H^1(\Gamma_n)\cong\Z^2$. 
Nevertheless, with some extra effort, 
one can establish the $H^1$-rigidity of $\Gamma\curvearrowright\R_\Gamma$. 

\begin{lemma}\label{H1ofH:rank=2}
When $\rank\Gamma=2$, 
$\Gamma\curvearrowright\R_\Gamma$ is $H^1$-rigid. 
In other words, the natural homomorphism $\H\to\Gamma$ induces 
an isomorphism between $H^1(\H)$ and $H^1(\Gamma)=\Hom(\Gamma,\Z)$. 
\end{lemma}

\begin{proof}
We may assume $\rank\Gamma_n=2$ (i.e.\ $\Gamma_n\cong\Z^2$) 
for all $n\in\N$. 
By Lemma \ref{H1ofHn} (2), 
$H^1(\H_n)$ is isomorphic to 
$H_0(\H_n)\cong H_0(\Gamma_n,C_c(\R_\Gamma,\Z))$. 
Moreover, by the proof of Lemma \ref{H1ofHn} (2), 
there exists a short exact sequence 
\[
\xymatrix@M=8pt{
0 \ar[r] & H_1(\Gamma_n) \ar[r] & 
H_0(\H_n) \ar[r] & 
\Ker\left(\bigoplus_{\Gamma/\Gamma_n}\Z\to\Z\right) \ar[r] & 0. 
}
\]
Let $T_n\subset\Gamma$ be 
a set of coset representatives for $\Gamma_n$. 
We assume $0\in T_n$. 
Since $H_1(\Gamma_n)$ is naturally isomorphic to $\Gamma_n$, 
we get the isomorphism  
\[
H_0(\H_n)\cong \Gamma_n\oplus\bigoplus_{T_n\setminus\{0\}}\Z. 
\]
This isomorphism sends $s\in\Gamma_n$ in the right-hand side to 
the equivalence class of the function 
\[
1_{[0_+,\infty)}-1_{[s_+,\infty)}\in C_c(\R_\Gamma,\Z), 
\]
and sends $1$ of the $t\in T_n\setminus\{0\}$ summand to 
the equivalence class of the function 
\[
1_{[0_+,\infty)}-1_{[t_+,\infty)}\in C_c(\R_\Gamma,\Z). 
\]
Next, let us look at the isomorphism $H^1(\H_n)\cong H_0(\H_n)$. 
Consider the cap product 
\[
\frown\ :H_1(\H_n)\times H^1(\H_n)\to H_0(\H_n). 
\]
Pick a generator $c_n\in H_1(\H_n)\cong H_2(\Gamma_n)\cong\Z$. 
Then the cap product 
\[
c_n\frown\cdot:H^1(\H_n)\to H_0(\H_n)
\]
gives the isomorphism. 

Suppose that a homomorphism $\xi:\H\to\Z$ is given. 
Let $\xi_n$ be the restriction of $\xi$ to $\H_n$. 
There exist $s\in\Gamma_1$, 
a finite subset $T'_1\subset T_1\setminus\{0\}$ 
and integers $(a_t)_{t\in T'_1}$ such that 
\[
c_1\frown[\xi_1]
=\left[1_{[0_+,\infty)}-1_{[s_{+},\infty)}\right]
+\sum_{t\in T'_1}a_t\left[1_{[0_+,\infty)}-1_{[t_+,\infty)}\right]
\in H_0(\H_1). 
\]
Let $\iota_n:\H_1\to\H_n$ be the inclusion map, 
so that $H^1(\iota_n)([\xi_n])=[\xi_1]$. 
There exists $m_n\in\Z$ such that $H_1(\iota_n)(c_1)=m_nc_n$. 
Notice that $\Gamma_1$ has index $\lvert m_n\rvert$ in $\Gamma_n$. 
By virtue of Proposition \ref{capproduct}, we have 
\begin{align*}
m_n\cdot\left(c_n\frown[\xi_n]\right)
&=H_1(\iota_n)(c_1)\frown[\xi_n] \\
&=H_0(\iota_n)\left(c_1\frown H^1(\iota_n)([\xi_n])\right) \\
&=H_0(\iota_n)\left(c_1\frown[\xi_1]\right). 
\end{align*}
Suppose that $n$ is large enough so that $T'_1\subset\Gamma_n$. 
Then one has 
\[
s':=s+\sum_{t\in T'_1}a_tt\in\Gamma_n, 
\]
and so 
\[
H_0(\iota_n)\left(c_1\frown[\xi_1]\right)
=\left[1_{[0_+,\infty)}-1_{[s'_{+},\infty)}\right]
\in\Gamma_n\cong H_1(\Gamma_n). 
\]
Hence $c_n\frown[\xi_n]$ lives 
in the $\Gamma_n$ component of $H_0(\H_n)$, 
which means that $[\xi_n]$ comes from $H^1(\Gamma_n)$. 
So, $\xi_n$ is cohomologous to a homomorphism $\H_n\to\Gamma_n\to\Z$. 
The rest of the proof is exactly the same 
as that of Lemma \ref{H1ofH:rank>2}. 
Since the action of $\Gamma_n$ on $\R_\Gamma$ is minimal, 
we may conclude that $\xi:\H\to\Z$ is cohomologous to 
a homomorphism $\H\to\Gamma\to\Z$. 
\end{proof}

\begin{remark}
When $\rank\Gamma=1$, the rank of $\Gamma_n$ is also one. 
In this case the action of $\Gamma_n$ is not minimal, 
and the choice of the function $f_n$ is not unique 
in the proof of Lemma \ref{H1ofH:rank>2}. 
Indeed, $\H=\R_\Delta\rtimes\Gamma$ is an AF groupoid, 
and $H^1(\H)$ becomes much larger. 
\end{remark}

\begin{remark}
In \cite{Ta2304arXiv}, it was shown that 
for any countable dense subgroups $\Gamma_1$ and $\Gamma_2$ of $\R$, 
$\R_{\Gamma_i}\rtimes\Gamma_i$ are isomorphic to each other 
if and only if $\Gamma_i\curvearrowright\R_{\Gamma_i}$ are conjugate. 
\end{remark}

\section{Classification of Stein's groups}

In this section, we would like to give 
a classification result for Stein's group. 

Let $\Lambda\subset(0,\infty)$ be 
a finitely generated multiplicative subgroup 
isomorphic to $\Z^N$ as abelian groups. 
We let $\Gamma\subset\R$ be a countable $\Z\Lambda$-module, 
which is dense in $\R$. 
The groups $\Gamma$ and $\Gamma\rtimes\Lambda$ act on $\R_\Gamma$ 
as in Section 3. 
Let $\H:=\R_\Gamma\rtimes\Gamma$ be the transformation groupoid 
and let 
\[
\S=\S(\Gamma,\Lambda):=\R_\Gamma\rtimes(\Gamma\rtimes\Lambda)
=\H\rtimes\Lambda
\]
be the Stein groupoid (Definition \ref{Steingroupoid}). 
As mentioned in Section 3, 
the groupoid $\H$ admits a unique invariant measure 
up to scalar multiplication. 
We denote by $\pi:\S\to\Gamma\rtimes\Lambda$ the natural homomorphism. 

\begin{proposition}\label{Sisrigid}
Suppose $\rank\Gamma\geq2$. 
The action $\Gamma\rtimes\Lambda\curvearrowright\R_\Gamma$ is $H^1$-rigid, 
that is, $\pi:\S\to\Gamma\rtimes\Lambda$ induces 
an isomorphism between $H^1(\S)$ and 
$H^1(\Gamma\rtimes\Lambda)=\Hom(\Gamma\rtimes\Lambda,\Z)$. 
\end{proposition}

\begin{proof}
By Lemma \ref{H1ofH:rank>2} and Lemma \ref{H1ofH:rank=2}, 
$\Gamma\curvearrowright\R_\Gamma$ is $H^1$-rigid. 
Since $\S$ is a semi-direct product of $\H$ by $\Lambda$ and 
$\Lambda$ is isomorphic to $\Z^N$, 
the conclusion follows 
from the repeated use of Proposition \ref{H1ofsemidirect}. 
\end{proof}

It is easy to verify that 
$H^1(\Gamma\rtimes\Lambda)=\Hom(\Gamma\rtimes\Lambda,\Z)$ is 
isomorphic to the direct sum of $H^1(\Lambda)\cong\Z^N$ and 
\[
H^1(\Gamma)^\Lambda:=\{\tau\in\Hom(\Gamma,\Z)\mid
\text{$\tau$ is invariant under the action of $\Lambda$}\}. 
\]

\begin{example}\label{algebraic}
Let us consider the case 
where $\Lambda\setminus\{1\}$ contains an algebraic number. 
Suppose that $\tau\in\Hom(\Gamma,\Z)$ is 
invariant by the action of $\Lambda$. 
Pick an algebraic number $\lambda\in\Lambda\setminus\{1\}$ and 
let $f(x)$ be its minimal polynomial over $\Z$. 
For any $t\in\Gamma$ and $k\in\N$, 
one has $\tau(t)=\tau(t\lambda^k)$. 
Therefore 
\[
0=\tau(tf(\lambda))=\tau(t)f(1), 
\]
which implies $\tau(t)=0$ because $1$ is not a root of $f$. 
Thus there exists no such $\tau$ except for zero. 
Thus, we have $H^1(\Gamma\rtimes\Lambda)=H^1(\Lambda)\cong\Z^N$. 
\end{example}

When $\Lambda\setminus\{1\}$ consists only of transcendental numbers, 
$H^1(\Gamma)^\Lambda$ may be non-trivial. 

\begin{example}\label{transcendental}
Let $\lambda>0$ be a transcendental number. 
Let $\Lambda:=\langle\lambda\rangle$ and 
$\Gamma:=\Z[\lambda,\lambda^{-1}]$. 
We can define $\tau\in H^1(\Gamma)^\Lambda$ 
by $\tau(\lambda^k)=1$ for all $k\in\Z$. 
Thus, $H^1(\Gamma)^\Lambda$ is isomorphic to $\Z$. 
\end{example}

Define 
\[
\Gamma_0:=\left\{t\in\Gamma\mid
\tau(t)=0\quad\forall\tau\in H^1(\Gamma)^\Lambda\right\}. 
\]
When $\Lambda\setminus\{1\}$ contains an algebraic number, 
$\Gamma_0$ equals $\Gamma$, because $H^1(\Gamma)^\Lambda=0$. 
When $\Lambda\setminus\{1\}$ consists only of transcendental numbers, 
$t-\lambda t$ is in $\Gamma_0$ 
for any $t\in\Gamma$ and $\lambda\in\Lambda\setminus\{1\}$. 
Either way, the action of $\Gamma_0$ on $\R_\Gamma$ is minimal. 
We remark that, 
for any $\zeta\in\Hom(\Gamma\rtimes\Lambda,\Z^N)$, 
the subgroupoid $\Ker(\zeta\circ\pi)\subset\S$ contains 
$\R_\Gamma\rtimes\Gamma_0$. 
So, if $\zeta$ is surjective, 
then the skew product $\S\times_{\zeta\circ\pi}\Z^N$ is minimal 
by Lemma \ref{skewisminimal}. 

\begin{lemma}\label{conditiononzeta}
Suppose $\rank\Gamma\geq2$. 
For $\zeta\in\Hom(\Gamma\rtimes\Lambda,\Z^N)$, 
the following conditions are equivalent. 
\begin{enumerate}
\item The subgroupoid $\Ker(\zeta\circ\pi)\subset\S$ admits 
an invariant measure. 
\item $\zeta$ factors through $\Lambda$ and $\zeta|\Lambda$ is injective. 
In other words, $\Ker\zeta=\Gamma\times\{1\}$. 
\item $\Ker(\zeta\circ\pi)=\H$. 
\end{enumerate}
\end{lemma}

\begin{proof}
(1)$\implies$(2)\:
As observed above, 
the subgroupoid $\Ker(\zeta\circ\pi)$ contains $\R_\Gamma\rtimes\Gamma_0$, 
which has an invariant measure $\mu$. 
Let $g=(x,t,\lambda)\in\Ker(\zeta\circ\pi)$. 
Choose a compact open neighborhood $O\subset\R_\Gamma$ of $x$, 
and set $U:=O\times\{(t,\lambda)\}$. 
Then $U$ is a bisection in $\Ker(\zeta\circ\pi)$ 
such that $\mu(s(U))=\lambda^{-1}\mu(r(U))$. 
Hence we get $\lambda=1$. 
In particular, the homomorphism 
\[
\Lambda\ni\lambda\mapsto\zeta(0,\lambda)\in\Z^N
\]
is injective, and so its image is of finite index in $\Z^N$. 

What remains to be proved is that $\zeta$ factors through $\Lambda$. 
Assume that 
$(t,1)\in\Gamma\rtimes\Lambda$ does not belong to $\Ker\zeta$. 
Since the image of $\zeta|\Lambda$ is of finite index, 
there exists $m\in\N$ and $\lambda\in\Lambda$ 
such that $\zeta(mt,1)=\zeta(0,\lambda)$. 
Thus, $\zeta(mt,\lambda^{-1})=0$, which implies $\lambda=1$. 
Therefore we get $\zeta(mt,1)=0$, which is a contradiction. 

(2)$\implies$(3) and (3)$\implies$(1) are obvious. 
\end{proof}

\begin{definition}
Let $\G$ be an ample groupoid and 
let $\mu$ be a measure on $\G^{(0)}$. 
\begin{enumerate}
\item For $g\in\G$ and $\mu\in(0,\infty)$, we write 
\[
g_*\,\mathrm{d}\mu=\lambda\,\mathrm{d}\mu
\]
if there exists a compact open bisection $U\subset\G$ 
such that $g\in U$ and 
$\mu(r(V))=\lambda\mu(s(V))$ holds 
for all clopen subset $V\subset U$. 
\item We set 
\[
R(\G,\mu):=\{\lambda\in(0,\infty)\mid
g_*\,\mathrm{d}\mu=\lambda\,\mathrm{d}\mu
\quad\text{for some }g\in\G\}. 
\]
and call it the ratio set for $(\G,\mu)$. 
\end{enumerate}
\end{definition}

\begin{remark}
In the definition above, 
it may be better to define the ratio set as 
the intersection of all $R(\G|Y,\mu)$ 
where $Y$ runs over all nonempty clopen subsets of $\G^{(0)}$. 
But we adopt that simplified version here, 
because it is sufficient for our use. 
\end{remark}

\begin{proposition}\label{recoverLambda}
Suppose $\rank\Gamma\geq2$. 
Let $\S=\S(\Gamma,\Lambda)$ be the Stein groupoid. 
Suppose that 
a homomorphism $\xi:\S\to\Z^N$ satisfies the following. 
\begin{enumerate}
\item The subgroupoid $\Ker\xi$ admits an invariant measure. 
\item The essential range of $\xi$ is $\Z^N$. 
\end{enumerate}
Then, $\Ker\xi$ is Kakutani equivalent to $\H$, 
and admits a unique invariant measure up to scalar multiplication. 
Moreover, 
letting $\nu$ be a invariant measure for $\Ker\xi$, 
we have $R(\S,\nu)=\Lambda$. 
\end{proposition}

\begin{proof}
According to Proposition \ref{Sisrigid}, 
there exists $\zeta\in\Hom(\Gamma\rtimes\Lambda,\Z^N)$ 
such that $\xi$ and $\zeta\circ\pi$ are cohomologous. 
There exists $f\in C(\R_\Gamma,\Z^N)$ such that 
\[
\xi(g)=\zeta(\pi(g))+f(r(g))-f(s(g))
\]
holds for all $g\in\S$. 
Since the essential range of $\xi$ is $\Z^N$, 
the same is true for $\zeta\circ\pi$, 
and so $\zeta$ is surjective. 
Hence the skew product $\S\times_{\zeta\circ\pi}\Z^N$ is minimal. 
It follows from Lemma \ref{Kakutani} (2) that 
$\Ker\xi$ and $\Ker(\zeta\circ\pi)$ are Kakutani equivalent. 
In particular, $\Ker(\zeta\circ\pi)$ admits an invariant measure. 
Thanks to Lemma \ref{conditiononzeta}, 
$\zeta$ factors through $\Lambda$, 
$\zeta|\Lambda$ is an isomorphism from $\Lambda$ to $\Z^N$ 
and $\Ker(\zeta\circ\pi)=\H$. 
In particular, 
$\Ker(\zeta\circ\pi)=\H$ has a unique invariant measure, 
and so does $\Ker\xi$. 

Let $\mu$ be an $\H$-invariant measure on $\R_\Gamma$. 
We would like to find a $\Ker\xi$-invariant measure on $\R_\Gamma$. 
In what follows, 
$\Lambda$ is identified with the subgroup of $\Gamma\rtimes\Lambda$. 
Let $\omega:\Z^N\to\Lambda$ be the inverse of $\zeta|\Lambda$. 
For any $g=(x,t,\lambda)\in\Ker\xi$, one has 
\[
g_*\,\mathrm{d}\mu=\lambda\,\mathrm{d}\mu
\]
and 
\[
\lambda=(0,\lambda)=\pi(g)=\omega(\zeta(\pi(g)))
=\omega(-f(r(g))+f(s(g))), 
\]
which implies 
\[
g_*\,\mathrm{d}\mu=\omega(f(r(g)))^{-1}\omega(f(s(g)))\,\mathrm{d}\mu. 
\]
Hence the measure $\nu$ defined by 
\[
\mathrm{d}\nu(x)=\omega(f(x))^{-1}\,\mathrm{d}\mu(x)
\]
is $\Ker\xi$-invariant. 

In order to compute the ratio set $R(S,\nu)$, 
we pick $g=(x,t,\lambda)\in\S$. 
Then 
\[
g_*\,\mathrm{d}\nu
=\lambda\omega(f(r(g)))\omega(f(s(g)))^{-1}\,\mathrm{d}\nu, 
\]
and so we get $R(S,\nu)\subset\Lambda$. 
For any $\lambda\in\Lambda$, letting $g:=(0_+,0,\lambda)\in\S$, 
we have 
\[
g_*\,\mathrm{d}\nu
=\lambda\omega(f(0_+))\omega(f(0_+))^{-1}\,\mathrm{d}\nu
=\lambda\,\mathrm{d}\nu, 
\]
which means $\Lambda\subset R(S,\nu)$. 
The proof is completed. 
\end{proof}

In general, when $\G$ is an ample groupoid and 
$\mu$ is a $\G$-invariant measure on $\G^{(0)}$, 
one can define a homomorphism $\tilde\mu:H_0(\G)\to\R$ 
by letting 
\[
\tilde\mu([f]):=\int f\ \mathrm{d}\mu
\]
for every $f\in C_c(\G^{(0)},\Z)$. 

\begin{theorem}\label{classifygroupoid}
For $i=1,2$, 
let $\Lambda_i\subset(0,\infty)$ be 
a finitely generated multiplicative subgroup and 
let $\Gamma_i\subset\R$ be a countable $\Z\Lambda_i$-module, 
which is dense in $\R$. 
Suppose $\rank\Gamma_i\geq2$ for $i=1,2$. 
Let $\S_i:=\S(\Gamma_i,\Lambda_i)$ be the Stein groupoids. 
The following are equivalent. 
\begin{enumerate}
\item $\S_1$ and $\S_2$ are isomorphic. 
\item $\Lambda_1=\Lambda_2$ and 
there exists $s>0$ such that $\Gamma_1=s\Gamma_2$. 
\item $(\Gamma_1\rtimes\Lambda_1)\curvearrowright\R_{\Gamma_1}$ 
and $(\Gamma_2\rtimes\Lambda_2)\curvearrowright\R_{\Gamma_2}$ 
are conjugate. 
\end{enumerate}
\end{theorem}

\begin{proof}
(2)$\implies$(3) is clear 
because the multiplication by $s$ gives a conjugacy 
between $(\Gamma_1\rtimes\Lambda_1)\curvearrowright\R_{\Gamma_1}$ 
and $(\Gamma_2\rtimes\Lambda_2)\curvearrowright\R_{\Gamma_2}$. 

(3)$\implies$(2) is obvious. 

(1)$\implies$(2)\:
Let $\Phi:\S_1\to\S_2$ be an isomorphism. 
By Lemma \ref{isotropy}, 
$\Lambda_1\cong\Lambda_2\cong\Z^N$ for some $N\in\N$. 
For each $i=1,2$, we let $\H_i:=\R_{\Gamma_i}\rtimes\Gamma_i$ and 
let $\mu_i$ be the $\H_i$-invariant measure 
satisfying $\Im\tilde\mu_i=\Gamma_i$. 
Choose an isomorphism $\zeta:\Lambda_1\to\Z^N$ and 
define the homomorphism $\xi:\S_1\to\Z^N$ 
by $\xi(x,t,\lambda):=\zeta(\lambda)$ for every $(x,t,\lambda)\in\S_1$. 
Clearly $\xi$ satisfies the assumption of the proposition above. 
Hence $\xi\circ\Phi^{-1}:\S_2\to\Z^N$ also satisfies 
the assumption of the proposition above. 
Therefore, we obtain 
\[
\Lambda_1=R(\S_1,\mu_1)=R(\S_2,\Phi_*(\mu_1))=\Lambda_2. 
\]
Moreover, $\H_1=\Ker\xi$ is isomorphic to $\Ker(\xi\circ\Phi^{-1})$, 
and $\Ker(\xi\circ\Phi^{-1})$ is Kakutani equivalent to $\H_2$ 
by Proposition \ref{recoverLambda}. 
It follows that $\H_1$ is Kakutani equivalent to $\H_2$. 
Then there exists $s>0$ such that $\Im\tilde\mu_1=s\Im\tilde\mu_2$, 
and so $\Gamma_1=s\Gamma_2$. 
\end{proof}

\begin{remark}\label{OErigidity}
As shown in \cite[Theorem 1.2]{Li18ETDS}, 
two transformation groupoids are isomorphic if and only if 
the two actions are continuously orbit equivalent. 
Thus, the theorem above means that 
if the two systems 
$(\Gamma_i\rtimes\Lambda_i)\curvearrowright\R_{\Gamma_i}$ are 
continuously orbit equivalent, 
then they are actually conjugate. 
This is a new example of 
continuous orbit equivalence rigidity phenomena. 
\end{remark}

Now we are ready to prove the main theorem. 

\begin{theorem}\label{classifygroup}
Let $\S_i:=\S(\Gamma_i,\Lambda_i)$ be the Stein groupoids 
as in Theorem \ref{classifygroupoid}. 
Let $\ell_i\in\Gamma_i\cap(0,\infty)$. 
The following conditions are equivalent. 
\begin{enumerate}
\item $V(\Gamma_1,\Lambda_1,\ell_1)$ is isomorphic to 
$V(\Gamma_2,\Lambda_2,\ell_2)$ as discrete groups. 
\item $D(V(\Gamma_1,\Lambda_1,\ell_1))$ is isomorphic to 
$D(V(\Gamma_2,\Lambda_2,\ell_2))$ as discrete groups. 
\item $\S(\Gamma_1,\Lambda_1)|[0_+,\ell_{1-}]$ is isomorphic to 
$\S(\Gamma_2,\Lambda_2)|[0_+,\ell_{2-}]$ as ample groupoids. 
\item $\Lambda_1=\Lambda_2$ and 
there exists $s>0$ such that $\Gamma_1=s\Gamma_2$ 
and $\ell_1-s\ell_2$ is zero in $H_0(\Lambda_1,\Gamma_1)$. 
\end{enumerate}
\end{theorem}

\begin{proof}
The equivalence between (1), (2) and (3) follows 
from Lemma \ref{groupVSgroupoid} 
(without assuming that $\Lambda_i$ is finitely generated 
and the rank of $\Gamma_i$ is greater than one). 

(3)$\implies$(4)\:
Assume that 
$\S_1|[0_+,\ell_{1-}]$ is isomorphic to $\S_2|[0_+,\ell_{2-}]$ 
as ample groupoids. 
For each $i=1,2$, 
$\S_i$ is canonically isomorphic to 
the product of $\S_i|[0_+,\ell_{i-}]$ and 
the full equivalence relation $\mathcal{R}$ on $\Z$, 
i.e.\ $\mathcal{R}=\Z\times\Z$. 
It follows that 
there exists an isomorphism $\Phi:\S_1\to\S_2$ 
such that $\Phi([0_+,\ell_{1-}])=[0_+,\ell_{2-}]$. 
In particular, 
Theorem \ref{classifygroupoid} implies $\Lambda_1=\Lambda_2$. 
Let $\H_i$, $\mu_i$ and $\xi:\S_1\to\Z^N$ be 
as in the proof of Theorem \ref{classifygroupoid}. 
Put $\nu:=\Phi_*(\mu_1)$, 
which is a $\Ker(\xi\circ\Phi^{-1})$-invariant measure. 
The following diagram is commutative: 
\[
\xymatrix@M=8pt{
\Gamma_1 \ar@{=}[d] & 
H_0(\H_1) \ar[l]^{\cong}_{\tilde\mu_1} \ar[r] \ar[d]_{H_0(\Phi)} & 
H_0(\S_1) \ar[d]_{H_0(\Phi)} \\
\Gamma_1 & 
H_0(\Ker(\xi\circ\Phi^{-1})) \ar[l]^-{\tilde\nu}_-{\cong} \ar[r] & 
H_0(\S_2)
}
\]
In the same way as Theorem \ref{classifygroupoid}, 
$\Ker(\xi\circ\Phi^{-1})$ is Kakutani equivalent to $\H_2$. 
Moreover, by the proof of Proposition 5.7 (2), 
we can find a continuous map $\omega:\R_{\Gamma_2}\to\Lambda_2$ and 
$s>0$ such that 
\[
\mathrm{d}\nu(x)=s\omega(x)\,\mathrm{d}\mu_2(x), 
\]
which implies 
\[
\Gamma_1=\Im\tilde\nu
=\left\{s\int f(x)\omega(x)\ \mathrm{d}\mu_2(x)
\mid f\in C_c(\R_{\Gamma_2},\Z)\right\}=s\Gamma_2. 
\]
Furthermore 
\[
s^{-1}\ell_1=s^{-1}\mu_1([0_+,\ell_{1-}])
=s^{-1}\nu([0_+,\ell_{2-}])
=\int 1_{[0_+,\ell_{2-}]}(x)\omega(x)\ \mathrm{d}\mu_2(x). 
\]
which means 
$s^{-1}\ell_1-\ell_2$ is trivial in $H_0(\Lambda_2,\Gamma_2)$, 
and hence $\ell_1-s\ell_2$ is trivial in $H_0(\Lambda_1,\Gamma_1)$. 

(4)$\implies$(3)\:
Multiplication by $s$ gives an isomorphism 
between $\S(\Gamma_2,\Lambda_2)|[0_+,\ell_{2-}]$ 
and $\S(s\Gamma_2,\Lambda_2)|[0_+,s\ell_{2-}]
=\S(\Gamma_1,\Lambda_1)|[0_+,s\ell_{2-}]$. 
As $1_{[0_+,\ell_{1-}]}$ and $1_{[0_+,s\ell_{2-}]}$ are equivalent 
in $H_0(\S(\Gamma_1,\Lambda_1))$, 
the conclusion follows. 
\end{proof}

\section{The case of rank one}

In our main results 
(Theorem \ref{classifygroupoid} and Theorem \ref{classifygroup}), 
the additive subgroup $\Gamma\subset\R$ has to be of rank at least two. 
Indeed, when the rank of $\Gamma$ is one, 
the $H^1$-rigidity does not hold. 
We would like to observe this closely through a concrete example. 

To start with, 
we briefly recall the notion of SFT groupoids from \cite{MM14Kyoto}. 
Let $A=[A(i,j)]_{i,j=0}^{N-1}$ be an $N\times N$ matrix 
with entries in $\{0,1\}$, where $N\in\N\setminus\{1\}$. 
We assume that $A$ is irreducible and not a permutation matrix. 
Define 
\[
X_A:=\left\{(x_n)_{n\in\N}\in\{0,\dots,N{-}1\}^\N\mid 
A(x_n,x_{n+1})=1\quad\forall n\in\N\right\}. 
\]
The shift transformation $\sigma_A$ on $X_A$ 
defined by $\sigma_A((x_n)_n)=(x_{n+1})_n$ is 
a continuous surjective map on $X_A$. 
The topological dynamical system $(X_A,\sigma_A)$ is called 
the one-sided shift of finite type for $A$. 

The SFT groupoid $\G_A$ for $(X_A,\sigma_A)$ is given by 
\[
\G_A:=\left\{(x,n,y)\in X_A\times\Z\times X_A\mid
\exists k,l\geq0,\ n=k{-}l,\ \sigma_A^k(x)=\sigma_A^l(y)\right\}. 
\]
The topology of $\G_A$ is generated by sets of the form 
\[
\left\{(x,k{-}l,y)\in\G_A
\mid x\in V,\ y\in W,\ \sigma_A^k(x)=\sigma_A^l(y)\right\}, 
\]
where $V,W\subset X_A$ are open and $k,l\geq0$. 
Two elements $(x,n,y)$ and $(x',n',y')$ in $\G_A$ are composable 
if and only if $y=x'$, and the multiplication and the inverse are 
\[
(x,n,y)\cdot(y,n',y')=(x,n{+}n',y'),\quad (x,n,y)^{-1}=(y,-n,x). 
\]
The range and source maps are given 
by $r(x,n,y)=(x,0,x)$ and $s(x,n,y)=(y,0,y)$. 
We identify $X_A$ with the unit space $\G_A^{(0)}$ via $x\mapsto(x,0,x)$. 
The groupoid $\G_A$ is minimal and essentially principal. 

One can define a homomorphism $\xi:\G_A\to \Z$ by $\xi(x,n,y):=n$. 
It is known that the subgroupoid $\Ker\xi$ is an AF groupoid. 
When the matrix $A$ is primitive, 
the AF groupoid $\Ker\xi$ is minimal and 
has a unique invariant probability measure $\mu$. 
Furthermore, 
the ratio set $R(\G_A,\mu)$ is equal to 
$\langle\lambda\rangle=\{\lambda^k\mid k\in\Z\}$, 
where $\lambda>0$ is the Perron-Frobenius eigenvalue of $A$. 

The following is also well-known. 
See \cite[Proposition 3.4]{MM14Kyoto} for instance. 

\begin{lemma}
The cohomology group $H^1(\G_A)$ is isomorphic to 
\[
C(X_A,\Z)/\{f-f\circ\sigma_A\mid f\in C(X_A,\Z)\}. 
\]
\end{lemma}

Now we consider Stein groupoids. 
For $\lambda>1$, we set 
$\S_\lambda:=\S(\Z[\lambda,\lambda^{-1}],\langle\lambda\rangle)$. 
To simplify notation, 
we write $\R_\lambda$ instead of $\R_{\Z[\lambda,\lambda^{-1}]}$. 
When $\lambda$ is irrational, 
Theorem \ref{Sisrigid}, Example \ref{algebraic} and 
Example \ref{transcendental} imply 
\[
H^1(\S_\lambda)\cong\begin{cases}\Z&\text{$\lambda$ is algebraic}\\
\Z\oplus\Z&\text{$\lambda$ is transcendental. }\end{cases}
\]
Let $n\in\N\setminus\{1\}$ 
and let $A_n$ be the $n\times n$ matrix whose entries are all one. 
The system $(X_{A_n},\sigma_{A_n})$ is 
the one-sided full shift over $n$ symbols. 
The $n$-adic expansion of real numbers induces 
a homeomorphism $\phi:\R_n\cap[0_+,1_-]\to X_{A_n}$ satisfying 
\[
q(x)=\sum_{i=1}^\infty\frac{\phi(x)_i}{n^i}
\]
for all $x\in\R_n\cap[0_+,1_-]$, 
where $q:\R_n\to\R$ is the canonical surjection. 
Observe that $\phi$ is order preserving, 
where the order on $\R_n$ is as described in Section 3 
and the order on $X_{A_n}$ is the lexicographical order 
(see \cite[Section 2]{KMW98ETDS}). 
Furthermore, $\phi$ gives rise to an isomorphism $\Phi$ 
from the reduction of the Stein groupoid $\S_n|[0_+,1_-]$ to 
the SFT groupoid $\G_{A_n}$. 
In particular, 
the cohomology group $H^1(\S_n)$ is isomorphic to $H^1(\G_{A_n})$. 

\begin{proposition}
For $n\in\N\setminus\{1\}$, 
we let $\S_n:=\S(\Z[1/n],\langle n\rangle)$. 
There exists an isomorphism $\Phi:\S_n|[0_+,1_-]\to\G_{A_n}$. 
In particular, $H^1(\S_n)$ is isomorphic to 
\[
C(X_{A_n},\Z)/\{f-f\circ\sigma_{A_n}\mid f\in C(X_{A_n},\Z)\}. 
\]
\end{proposition}

Evidently $H^1(\S_n)$ is larger than 
$H^1(\langle n\rangle)=H^1(\Z)\cong\Z$, 
and so $\Z[1/n]\rtimes\langle n\rangle\curvearrowright\R_n$ 
is not $H^1$-rigid. 

\begin{problem}
Let $\lambda>1$. 
When $\lambda$ is in $\Q\setminus\N$, compute $H^1(\S_\lambda)$. 
\end{problem}

When $\lambda=p/q$ with $p,q\in\N$ coprime, 
it is known that the homology groups of $\S_\lambda$ are 
\[
H_k(\S_\lambda)=\begin{cases}\Z/(p-q)\Z & k=0\\ 0 & k>0. \end{cases}
\]
See \cite[Proposition 2.17]{CPPR11JFA} and \cite{Li15JFA}. 

Next we take a closer look at $\S_2$. 
As observed above, 
there exists an isomorphism $\Phi_1:\S_2|[0_+,1_-]\to\G_{A_2}$. 
Set 
\[
B:=\begin{bmatrix}1&1\\1&0\end{bmatrix}. 
\]
Notice that $B$ is primitive and 
its Perron-Frobenius eigenvalue is $\beta:=(\sqrt{5}+1)/2$. 
The shift space $X_B\subset\{0,1\}^\N$ consists of 
infinite sequences of $\{0,1\}$ in which `11' do not appear. 
We can define a homeomorphism from $X_{A_2}=\{0,1\}^\N$ to $X_B$ 
by substituting the letter `1' to the word `10', 
and it gives rise to an isomorphism $\Phi_2:\G_{A_2}\to\G_B$ 
(see \cite[Section 5]{Ma10PJM}). 
Consider the homomorphism $\xi:\G_B\to\Z$ defined by $\xi(x,n,y):=n$. 
Then $\Ker\xi$ is a minimal AF groupoid 
with a unique invariant probability measure $\mu$ 
and $R(\G_B,\mu)=\langle\beta\rangle$. 
Since $\Phi_2\circ\Phi_1$ is an isomorphism 
from $\S_2|[0_+,1_-]$ to $\G_B$, 
the cocycle $\xi\circ\Phi_2\circ\Phi_1$ has the same property. 
Hence there exists a cocycle $\xi':\S_2\to\Z$ such that 
$\Ker\xi'$ is a minimal AF groupoid 
with a unique invariant measure $\mu'$ up to scalar multiplication, 
and $R(\S_2,\mu')=\langle\beta\rangle$. 
This contrasts to Proposition \ref{recoverLambda}. 
It is then natural to ask the following. 

\begin{problem}
Let $\xi:\S_2\to\Z$ be a nonzero homomorphism. 
\begin{enumerate}
\item Determine when the subgroupoid $\Ker\xi$ is AF. 
\item When $\Ker\xi$ admits a unique invariant measure 
up to scalar multiplication, we can consider the ratio set. 
Determine the groups which arise as such ratio sets. 
\end{enumerate}
\end{problem}

In the rest of this section, 
we would like to observe that $\S_2$ is embeddable into $\S_\beta$. 
The $\beta$-adic expansion of real numbers induces 
a homeomorphism $\psi:X_B\to\R_\beta\cap[0_+,1_-]$ satisfying 
\[
q(\psi(x))=\sum_{i=1}^\infty\frac{x_i}{\beta^i}
\]
for all $x\in X_B$, 
where $q:\R_\beta\to\R$ is the canonical surjection 
(see \cite[Example 2]{KMW98ETDS} and \cite[Example 2.3]{MM14Aust}). 
In the same way as mentioned before, $\psi$ is an order preserving map. 
Let us examine how the map $\psi$ transforms cylinder sets of $X_B$. 
When a word $w=w_1w_2\dots w_k\in\{0,1\}^k$ does not contain `11', 
we put the cylinder set 
\[
C(w):=\{x\in X_B\mid x_i=w_i\quad\forall i=1,2,\dots,k\}, 
\]
and let 
\[
b(w):=\sum_{i=1}^k\frac{w_i}{\beta^i}\in[0,1). 
\]
Then the image of $C(w)$ under $\psi$ is described as follows: 
if $w_k=0$, then 
\[
\psi(C(w))=\left[(b(w))_+,(b(w)+\beta^{-k})_-\right]\subset\R_\beta, 
\]
and if $w_k=1$, then 
\[
\psi(C(w))=\left[(b(w))_+,(b(w)+\beta^{-k-1})_-\right]\subset\R_\beta. 
\]
As a consequence, $\psi$ gives rise to 
a homomorphism $\Phi_3:\G_B\to\S_\beta|[0_+,1_-]$. 
\[
\xymatrix@M=8pt{
\S_2|[0_+,1_-] \ar[r]^-{\Phi_1}_-{\cong} & 
\G_{A_2} \ar[r]^{\Phi_2}_{\cong} & 
\G_B \ar@{^{(}-_>}[r]^-{\Phi_3} & 
\S_\beta|[0_+,1_-]
}
\]
The reader should be warned that $\Phi_3$ is not surjective. 
Indeed, as 
\[
\psi(010101\dots)=(\beta^{-1})_-
\quad\text{and}\quad 
\psi(101010\dots)=1_-, 
\]
$(1_-,\beta^{-2},1)\in\S_\beta$ does not belong to the image of $\Phi_3$. 
Also, one can see that 
the homeomorphism $\R_2\cap[0_+,1_-]\to\R_\beta\cap[0_+,1_-]$ 
induced by $\Phi:=\Phi_3\circ\Phi_2\circ\Phi_1$ preserves the order, 
i.e.\ $x<y$ if and only if $\Phi(x)<\Phi(y)$ 
for $x,y\in\R_2\cap[0_+,1_-]$. 

To sum up, we get the following. 

\begin{proposition}\label{S2<Sbeta}
Let $\beta:=(\sqrt{5}+1)/2$. 
There exists a continuous homomorphism 
$\Phi:\S_2\cap[0_+,1_-]\to\S_\beta\cap[0_+,1_-]$ 
satisfying the following. 
\begin{enumerate}
\item $\Phi$ is injective and open. 
\item The restriction of $\Phi$ to the unit spaces is 
a homeomorphism preserving the order, that is, 
$x<y$ if and only if $\Phi(x)<\Phi(y)$ 
for $x,y\in\R_2\cap[0_+,1_-]$. 
\end{enumerate}
Moreover, 
$\Phi$ naturally extends to an embedding $\S_2\hookrightarrow\S_\beta$, 
which also has the properties above. 
\end{proposition}

It follows from the proposition above that 
the Higman-Thompson group $V_{2,1}$ (resp.\ $F_{2,1}$) embeds 
into Cleary's group $V_{\beta,1}$ (resp.\ $F_{\beta,1}$). 
Here, $V_{\lambda,1}$ stands for 
$V(\Z[\lambda,\lambda^{-1}],\langle\lambda\rangle,1)$, 
and $F_{\lambda,1}$ is the subgroup of $V_{\lambda,1}$ 
consisting of order preserving maps. 
The group $F_{2,1}$ is known as Thompson's group, 
and the group $F_{\beta,1}$ is known as its irrational slope version. 

\begin{remark}
The embedding $V_{2,1}\hookrightarrow V_{\beta,1}$ is mentioned 
in \cite[Section 6.7.2]{Ma15crelle}. 
The embedding $F_{2,1}\hookrightarrow F_{\beta,1}$ is discussed 
in \cite{BNR21PublMat}. 
It is also known that $F_{\beta,1}$ does not embed into $F_{2,1}$ 
(\cite{HM23GGD}). 
\end{remark}

\begin{problem}
For $\lambda,\mu>1$, 
detemine when $\S_\lambda$ embeds into $\S_\mu$ 
in a way similar to Proposition \ref{S2<Sbeta}. 
Notice that, in view of Proposition \ref{recoverLambda} and its proof, 
if $\lambda$ is irrational, 
then $\mu$ must be equal to a power of $\lambda$. 
For example, when $\beta=(\sqrt{5}+1)/2$, 
\[
\S_{\beta^2}=\S(\Z[\beta^2,\beta^{-2}],\langle\beta^2\rangle)
=\S(\Z[\beta,\beta^{-1}],\langle\beta^2\rangle)
\]
clearly embeds into 
$\S_\beta=\S(\Z[\beta,\beta^{-1}],\langle\beta\rangle)$. 
\end{problem}

\section{Attracting elements}

We recall the notion of attracting elements from \cite{Ma16Adv} 
(see also \cite[Section 3]{MM14Kyoto}). 

\begin{definition}[{\cite[Definition 5.11]{Ma16Adv}}]
Let $\G$ be an ample groupoid and let $g\in \G'$, i.e. $r(g)=s(g)$. 
We say that $g$ is an attracting element 
if for any compact open neighborhood $V$ of $x$, 
there exists a compact open bisection $U\subset\G$ 
such that $g\in U\subset V$ and $r(U)$ is a proper subset of $s(U)$. 
For such $U$, we call the closed set 
\[
Y:=\bigcap_{n=1}^\infty r(U^n)
\]
a limit set of $g$. 
Notice that $r(g)$ is contained in every limit set of $g$. 
\end{definition}

It is easy to see the following. 

\begin{lemma}
Let $\S=\S(\Gamma,\Lambda)$ be a Stein groupoid and 
let $\pi:\S\to\Lambda$ be the natural homomorphism. 
Suppose $g\in\S'$. 
Then $g$ is attracting if and only if $\pi(g)<1$. 
\end{lemma}

By using the notion of attracting elements, 
we can strengthen Lemma \ref{isotropy} as follows. 

\begin{lemma}\label{orderpreserve}
For $i=1,2$, 
let $\S_i:=\S(\Gamma_i,\Lambda_i)$ be a Stein groupoid. 
Suppose that 
there exists a continuous homomorphism $\Phi:\S_1\to\S_2$ 
which is injective and open. 
Then, there exists a homomorphism $\phi:\Lambda_1\to\Lambda_2$ 
preserving the order, i.e.\ 
$\phi(\lambda)<\phi(\lambda')$ if and only if $\lambda<\lambda'$. 
\end{lemma}

\begin{proof}
For $i=1,2$, we let $\pi_i:\S_i\to\Lambda_i$ denote 
the natural homomorphism. 
By Lemma \ref{isotropy0} (2), 
there exists $x\in\S_1^{(0)}$ such that 
$\pi_1:(\S_1)_x\to\Lambda_1$ is an isomorphism. 
Clearly, $\Phi$ induces an injective homomorphism 
from $(\S_1)_x$ to $(\S_2)_{\Phi(x)}$. 
Besides, $g\in(\S_1)_x$ is attracting 
if and only if $\Phi(g)\in(\S_2)_{\Phi(x)}$ is attracting. 
By Lemma \ref{isotropy0} (1), 
$\pi_2:(\S_2)_{\Phi(x)}\to\Lambda_2$ is injective. 
Therefore 
\[
\phi:=\pi_2\circ\Phi\circ(\pi_1|(\S_1)_x)^{-1}
:\Lambda_1\to(\S_1)_x\to(\S_2)_{\Phi(x)}\to\Lambda_2
\]
is an injective homomorphism. 
It follow from the lemma above that 
$\phi(\lambda)<1$ if and only if $\lambda<1$. 
Thus $\phi$ is order preserving. 
\end{proof}

\begin{example}
The lemma above gives a restriction 
for two Stein groupoids being isomorphic. 
For example, let $\Lambda_1:=\langle 2,3\rangle\subset(0,\infty)$ 
(the multiplicative subgroup generated by $2$ and $3$) 
and let $\Lambda_2:=\langle 2,5\rangle\subset(0,\infty)$. 
Unless $\rank\Gamma_i\geq2$, 
we cannot use Theorem \ref{classifygroupoid}. 
But, there does not exist an order preserving embedding 
$\Lambda_1\to\Lambda_2$, and so, by using the lemma above, 
one can deduce that $\S_1$ is not isomorphic to $\S_2$ 
even in the case $\rank\Gamma_i=1$. 
Notice that the homology groups $H_k(\S_i)$ cannot distinguish 
the groupoids $\S_1$ and $\S_2$ in this case, 
because they are trivial for all $k$. 

When $\Lambda_1=\langle 2,3\rangle$ and $\Lambda_2:=\langle 2,9\rangle$, 
one cannot use this argument, 
and it is not known if the two Stein groupoids are isomorphic or not. 
\end{example}

\begin{problem}
Suppose that $\Lambda_i$ are subgroups of $\Q\cap(0,\infty)$ and 
$\Gamma_i$ are subgroups of $\Q$. 
Determine when the Stein groupoids $\S(\Gamma_1,\Lambda_1)$ and 
$\S(\Gamma_2,\Lambda_2)$ are isomorphic to each other. 
\end{problem}

\newcommand{\noopsort}[1]{}

\end{document}